\documentclass[a4paper]{article}
\usepackage[utf8]{inputenc}
\usepackage[english]{babel}
\usepackage{amsmath}
\usepackage{amsthm}
\usepackage{amssymb}
\usepackage{bbm}	
\usepackage[numbers]{natbib}	
\usepackage{graphicx, color, soul} 
\usepackage[title]{appendix}
\usepackage[paperwidth=19.5cm,paperheight=25cm]{geometry}
\usepackage{url}
\usepackage{blkarray} 
\usepackage[colorlinks=true,linkcolor=blue,citecolor=blue]{hyperref}
\usepackage[capitalize,nameinlink]{cleveref}

\usepackage{authblk} 

\usepackage[shortlabels]{enumitem}	
\setlist{labelsep=.25in,leftmargin=*,labelindent=1cm,topsep=2pt,noitemsep}	
\setlist[enumerate]{label=(\roman*)}

\numberwithin{equation}{section}

\theoremstyle{plain}
\newtheorem{theorem}{Theorem}[section]
\newtheorem{corollary}[theorem]{Corollary}
\newtheorem{proposition}[theorem]{Proposition}

\newtheorem*{theorem*}{Theorem}
\theoremstyle{remark}
\newtheorem*{lemma*}{Lemma}
\newtheorem{remark}[theorem]{Remark}
\theoremstyle{definition}
\newtheorem{example}[theorem]{Example}

\newcommand{\Polya}{P\'{o}lya }
\newcommand{\limN}{\lim_{n\rightarrow\infty}}

\newcommand{\sumN}{\sum_{i=1}^n}

\newcommand{\Be}{\textnormal{Beta}} 
\newcommand{\Dir}{\textnormal{Dir}} 
\newcommand{\DP}{\textnormal{DP}} 

\usepackage{geometry}
\title{Characterization of exchangeable measure-valued P\'{o}lya urn sequences}
\author[1,2]{ Hristo Sariev\thanks{h.sariev@math.bas.bg}}
\author[2,1]{Mladen Savov\thanks{msavov@fmi.uni-sofia.bg}}
\affil[1]{\normalsize Institute of Mathematics and Informatics, Bulgarian Academy of Sciences, 8 Acad. Georgi Bonchev Str., Sofia 1113, Bulgaria}
\affil[2]{\normalsize Faculty of Mathematics and Informatics, Sofia University "St. Kliment Ohridski", 5 James Bourchier Blvd, Sofia 1164, Bulgaria}

\date{}

\begin{document}

\maketitle
\begin{abstract}
Measure-valued \Polya urn sequences (MVPS) are a generalization of the observation processes generated by $k$-color \Polya urn models, where the space of colors $\mathbb{X}$ is a complete separable metric space and the urn composition is a finite measure on $\mathbb{X}$, in which case reinforcement reduces to a summation of measures. In this paper, we prove a representation theorem for the reinforcement measures $R$ of all exchangeable MVPSs, which leads to a characterization result for their directing random measures $\tilde{P}$. In particular, when $\mathbb{X}$ is countable or $R$ is dominated by the initial distribution $\nu$, then any exchangeable MVPS is a Dirichlet process mixture model over a family of probability distributions with disjoint supports. Furthermore, for all exchangeable MVPSs, the predictive distributions converge on a set of probability one in total variation to $\tilde{P}$. Importantly, we do not restrict our analysis to balanced MVPSs, in the terminology of $k$-color urns, but rather show that the only non-balanced exchangeable MVPSs are sequences of i.i.d. random variables.
\end{abstract}



\section{Introduction}

The classical two-color \Polya urn model, which describes the evolution of an urn reinforced with one ball of the observed color, has a fundamental role in the predictive construction of prior distributions for Bayesian inference \cite{fortini2012petrone}. Suppose we have an urn that initially contains $w_0,w_1>0$ balls of colors $0$ and $1$, respectively. Let us denote by $X_n\in\{0,1\}$ the color of the sampled ball at step $n$. Then $X_1=1$ with probability $\frac{w_1}{w_0+w_1}$ and, for every $n=1,2,\ldots$,
\begin{equation}\label{intro:predictive:polya_urn}
\mathbb{P}(X_{n+1}=1|X_1,\ldots,X_n)=\frac{w_1+\sumN X_i}{w_0+w_1+n}\qquad\mbox{a.s.},
\end{equation}
which is the proportion of $1$-balls in the urn at time $n$. It is well-known that the process $(X_n)_{n\geq1}$ is exchangeable, that is, its law is invariant under finite permutations of the indices, and satisfies, as $n\rightarrow\infty$,
\[\frac{1}{n}\sumN X_i\overset{a.s.}{\longrightarrow}\tilde{p},\]
where $\tilde{p}$ has a Beta distribution with parameters $(w_0,w_1)$, see, e.g., \citep[][p.8]{pemantle2007}; therefore, by de Finetti representation theorem for exchangeable sequences \citep[][Theorem 3.1]{aldous1985}, given $\tilde{p}$, the $X_n$ are conditionally independent and identically distributed with probability of "success" $\tilde{p}$. The converse is also true -- any exchangeable sequence of Bernoulli random variables with a Beta prior distribution has the predictive structure \eqref{intro:predictive:polya_urn}.

There has been much interest in generalizing the construction \eqref{intro:predictive:polya_urn} to obtain greater model flexibility and, at the same time, retain tractability of the process dynamics, see \cite{fortini2012petrone,mahmoud2008,pemantle2007} and references therein. Various extensions of the \Polya urn scheme consider time-dependent, randomized, or generalized reinforcement mechanisms where colors other than those observed are added to the urn. More recently, \cite{thacker2022,janson2019,mailler2017} have proposed an extended class of \textit{measure-valued \Polya urn processes} that are very general in nature, yet retain characteristic features of urn processes and contain, as a special case, most $k$-color urn models. The idea is to consider the urn composition as a finite measure $\mu$ on the space of colors $\mathbb{X}$, in the sense that, for any measurable set $B\subseteq\mathbb{X}$, the quantity $\mu(B)$ records the total \textit{mass} of balls in the urn whose colors lie in $B$. Reinforcement is then reduced to a summation of measures, so that the updated urn composition, given that a ball with color $x$ has been observed, becomes $\mu+R_x$, where $R$ is a (random) transition kernel on $\mathbb{X}$, i.e. a map $x\mapsto R_x$ from $\mathbb{X}$ to the space of (random) measures on $\mathbb{X}$. Assuming $\mu+R_x$ as the new urn composition, the next draw proceeds in the same way as before, independent of all previous draws. Thus, a sequence of finite measures $(\mu_n)_{n\geq0}$ is a \textit{measure-valued \Polya urn process} (MVPP) if it is a Markov process with the aforementioned additive structure for a given finite reinforcement kernel $R$. In this case, Theorem 1 in \cite{fortini2021} implies the existence of a companion observation process $(X_n)_{n\geq 1}$ such that $X_1\sim\frac{\mu_0}{\mu_0(\mathbb{X})}=:\mu_0'$ and, for each $n=1,2,\ldots$,
\[\mathbb{P}(X_{n+1}\in\cdot\mid \mu_0,X_1,\mu_1,\ldots,X_n,\mu_n)=\frac{\mu_n(\cdot)}{\mu_n(\mathbb{X})}=:\mu_n'(\cdot)\qquad\mbox{a.s.}\]
In particular, when the reinforcement $R$ is non-random, the above becomes
\begin{equation}\label{intro:predictive:mvps}
\mathbb{P}(X_{n+1}\in\cdot\mid X_1,\ldots,X_n)=\frac{\mu_0(\cdot)+\sumN R_{X_i}(\cdot)}{\mu_0(\mathbb{X})+\sumN R_{X_i}(\mathbb{X})}\qquad\mbox{a.s.},
\end{equation}
which corresponds to the normalized urn composition at time $n$ in the urn analogy. We call such a sequence $(X_n)_{n\geq1}$ a \textit{measure-valued \Polya urn sequence} (MVPS) to distinguish it from the process $(\mu_n)_{n\geq0}$.

Analysis of MVPPs has been centered around two very distinct cases. In \cite{fortini2021,sariev2023}, the authors study the so-called "diagonal" model, where only the observed color is reinforced, i.e. $R_x=w(x)\cdot\delta_x$, with $w(x)>0$, for $x\in\mathbb{X}$, and $\delta_x$ is the unit mass at $x$. It follows under certain conditions that there exists a random probability measure $\tilde{P}$ on $\mathbb{X}$ such that, as $n\rightarrow\infty$, $(i)$ the normalized urn composition $\mu_n'$ converges almost surely (a.s.) in total variation to $\tilde{P}$; and $(ii)$ $\tilde{P}$ is concentrated on a subset of so-called "dominant" colors. On the other hand, \cite{thacker2022,mailler2017,mailler2020} consider MVPPs for which there exists an underlying Markov chain with kernel $R$ that satisfies some "irreducibility"-type conditions (see \citep[][Section 1.2]{mailler2020} for more details). Then, under additional assumptions, they prove that $\mu_n'$ converges a.s. weakly to a (deterministic) probability measure on $\mathbb{X}$.

In this work, we consider MVPPs that generate an exchangeable observation sequence $(X_n)_{n\geq1}$ via \eqref{intro:predictive:mvps}. Standard results in exchangeability theory then guarantee that the predictive distributions of $(X_n)_{n\geq1}$ converge a.s. weakly to a random probability measure $\tilde{P}$ on $\mathbb{X}$, called the \textit{directing random measure}, whose distribution determines the law of the exchangeable process, see, e.g., \citep[][Section 3]{aldous1985}. Therefore, the only exchangeable "irreducible" MVPSs are sequences of i.i.d. random variables. Our goal is to characterize all $R$ for which $(X_n)_{n\geq1}$ is exchangeable, and in the process prove some new facts about $\tilde{P}$ and hence about the distribution of $(X_n)_{n\geq1}$. This will help us assess the limitations of the whole class of exchangeable MVPSs when used as models for Bayesian analysis. Important to our understanding are the \textit{kernel based Dirichlet sequences} studied by Berti et al. \cite{berti2023}, which are exchangeable MVPSs whose reinforcement $R$ is a regular version of the conditional distribution for $\mu_0'$ given some sub-$\sigma$-algebra $\mathcal{G}$,
\begin{equation}\label{intro:reinforcement:kds}
R_x(\cdot)=\mu_0'(\cdot\mid\mathcal{G})(x)\qquad\mbox{for }\mu_0'\mbox{-a.e. }x,
\end{equation}
see Section \ref{section:model} for a rigorous definition. This assumption greatly simplifies the subsequent analysis and allows \cite{berti2023} to obtain a complete characterization of the directing random measure of any kernel based Dirichlet sequence. In \citep[][p.18]{berti2023}, the authors tentatively raise the conjecture that \eqref{intro:reinforcement:kds} holds true for every exchangeable MVPS such that $R_x(\mathbb{X})=1$, not excluding the possibility of counterexamples. Our main theorem states that a normalized version of condition \eqref{intro:reinforcement:kds} is indeed true for every member of the class of exchangeable MVPSs, regardless of whether $R_x(\mathbb{X})=1$. That such a representation exists is not obvious, and its proof requires the use of techniques that go beyond those typical of the area and involve intensive use of abstract measure-valued objects. We also give additional results that do not make use of \eqref{intro:reinforcement:kds}. In particular, we prove that when $(X_n)_{n\geq1}$ is not i.i.d., then $R_x(\mathbb{X})$ is constant for almost every $x$, which implies that a "diagonal" MVPS will be exchangeable only if $w(x)$ is constant, see also Example \ref{example:model:polya_sequence}. In addition, we show that, for fixed $x$, the reinforcement $R_x$ is either absolutely continuous or mutually singular with respect to $\nu$. On the other hand, using a refined version of \eqref{intro:reinforcement:kds}, we provide a complete description of all possible exchangeable $k$-color urn models with positive time-homogeneous reinforcement, which to our knowledge has not been done before.

The rest of the paper is structured as follows. In Section \ref{section:model}, we provide notation, state some general facts about exchangeable sequences, and formally define our model. Results and examples are given in Section \ref{section:results}, with proofs postponed to Section \ref{section:proofs}. We state our representation theorem for general $R_x$ in Section \ref{section:results:general}, while in Section \ref{section:results:abs_cont} we study the case when $R_x$ is dominated by $\mu_0'$, which includes $k$-color urns, and in Section \ref{section:results:singular} the case when $R_x$ and $\mu_0'$ are mutually singular. A final section concludes the paper.
 
\section{The model}\label{section:model}

\subsection{Preliminaries}

Let $(\Omega,\mathcal{H},\mathbb{P})$ be a probability space, $\mathbb{X}$ a complete separable metric space, and $\mathcal{X}$ the associated Borel $\sigma$-algebra on $\mathbb{X}$. Standard results imply that $\mathcal{X}$ is countably generated. A \textit{transition kernel} on $\mathbb{X}$ is a function $R:\mathbb{X}\times\mathcal{X}\rightarrow\bar{\mathbb{R}}_+$ that satisfies $(i)$ the map $x\mapsto R(x,B)\equiv R_x(B)$ is $\mathcal{X}$-measurable, for all $B\in\mathcal{X}$; and $(ii)$ $B\mapsto R_x(B)$ is a measure on $\mathbb{X}$, for all $x\in\mathbb{X}$. Equivalently, $R$ can be represented as a measurable function from $\mathbb{X}$ to the space of measures on $\mathbb{X}$. In addition, a transition kernel $R$ is said to be \textit{finite} if $R_x(\mathbb{X})<\infty$, for all $x\in\mathbb{X}$, \textit{non null} if $R_x(\mathbb{X})>0$, for all $x\in\mathbb{X}$, and is called a \textit{probability kernel} if $R_x(\mathbb{X})=1$, for all $x\in\mathbb{X}$. A \textit{random (probability) measure} is a transition (probability) kernel $\tilde{P}:\Omega\times\mathcal{X}\rightarrow\bar{\mathbb{R}}_+$ from $\Omega$ to $\mathbb{X}$.

Let $\nu$ be a probability measure on $\mathcal{X}$, $R$ a transition probability kernel on $\mathbb{X}$, and $\mathcal{G}\subseteq\mathcal{X}$ a sub-$\sigma$-algebra on $\mathbb{X}$. Then $R$ is said to be a \textit{regular version of the conditional distribution (r.c.d.) for $\nu$ given $\mathcal{G}$}, which for emphasis we will denote by
\[R_x(\cdot)=\nu(\cdot\mid\mathcal{G})(x)\qquad\mbox{for }\nu\mbox{-a.e. }x,\]
if $(i)$ $x\mapsto R_x(B)$ is $\mathcal{G}$-measurable, for all $B\in\mathcal{X}$; and $(ii)$ $\int_AR_x(B)\nu(dx)=\nu(A\cap B)$, for all $A\in\mathcal{G}$ and $B\in\mathcal{X}$. It follows from the assumptions on $(\mathbb{X},\mathcal{X})$ that a r.c.d. for $\nu$ under $\mathcal{G}$ exists and is unique up to a $\nu$-null set.

Let $X$ be an $\mathbb{X}$-valued random variable with marginal distribution $X\sim\mathbb{P}_X$. Depending on the context, we would work with conditional distributions of the type $\mathbb{P}(\cdot\mid X)$ or $\mathbb{P}(\cdot\mid X=x)$, which are related by
\[\int_{\{X\in A\}}\mathbb{P}(B|X)(\omega)\mathbb{P}(d\omega)=\int_A\mathbb{P}(B|X=x)\mathbb{P}_X(dx),\]
for all $A\in\mathcal{X}$ and $B\in\mathcal{H}$.

Recall that an $\mathbb{X}$-valued sequence of random variables $(X_n)_{n\geq1}$ is exchangeable if
\[(X_1,\ldots,X_n)\overset{d}{=}(X_{\sigma(1)},\ldots,X_{\sigma(n)}),\]
for every $n\geq2$ and all permutations $\sigma$ of $\{1,\ldots,n\}$. In that case (see \citep[][Section 3]{aldous1985}), there exists a random probability measure $\tilde{P}$ on $\mathbb{X}$, called the \textit{directing random measure} of the sequence, such that, given $\tilde{P}$, the random variables $X_1,X_2,\ldots$ are conditionally independent and identically distributed (i.i.d.) with marginal distribution $\tilde{P}$,
\begin{eqnarray}
X_n\mid\tilde{P}& \overset{i.i.d.}{\sim} & \tilde{P} \nonumber\\
\tilde{P}& \sim & Q \nonumber
\end{eqnarray}
where $Q$ is the (nonparametric) \textit{prior} distribution of $\tilde{P}$. Moreover, for every $A\in\mathcal{X}$,
\begin{equation}\label{model:exchangeable:convergence}
\mathbb{P}(X_{n+1}\in A|X_1,\ldots,X_n)\overset{a.s.}{\longrightarrow}\tilde{P}(A).
\end{equation}

Since we take a predictive approach to model building, the following result from \cite{fortini2000}, which provides necessary and sufficient conditions for the system of predictive distributions of any stochastic process to be consistent with exchangeability, becomes our starting point of analysis.

\begin{theorem}[Theorem 3.1, Proposition 3.2 in \cite{fortini2000}]\label{model:predictive_characterization}
A sequence of random variables $(X_n)_{n\geq1}$ is exchangeable if and only if, for each $n=0,1,2,\ldots$ and every $A,B\in\mathcal{X}$,
\begin{equation}\label{model:exchangeable:predictive}
\mathbb{P}\bigl(X_{n+1}\in A,X_{n+2}\in B|X_1,\ldots,X_n\bigr)=\mathbb{P}\bigl(X_{n+1}\in B,X_{n+2}\in A|X_1,\ldots,X_n\bigr)\quad\mbox{a.s.},
\end{equation}
and
\begin{equation}\label{model:exchangeable:predictive_perm}
\mathbb{P}(X_{n+1}\in A|X_1=x_1,\ldots,X_n=x_n)=\mathbb{P}(X_{n+1}\in A|X_1=x_{\sigma(1)},\ldots,X_n=x_{\sigma(n)}),
\end{equation}
for all permutations $\sigma$ of $\{1,\ldots,n\}$ and almost every $(x_1,\ldots,x_n)\in\mathbb{X}^n$ with respect to the marginal distribution $\mathbb{P}_{(X_1,\ldots,X_n)}$ of $(X_1,\ldots,X_n)$ on $\mathcal{X}^n$, where the case $n=0$ should be understood as an unconditional statement.
\end{theorem}

\subsection{Exchangeable MVPS}

We will call any sequence $(X_n)_{n\geq1}$ of $\mathbb{X}$-valued random variables on $(\Omega,\mathcal{H},\mathbb{P})$ a \textit{measure-valued \Polya urn sequence} with parameters $\theta$, $\nu$ and $R$, denoted MVPS$(\theta,\nu,R)$, if $X_1\sim\nu$ and, for each $n=1,2,\ldots$,
\begin{equation}\label{model:predictive}
\mathbb{P}(X_{n+1}\in\cdot\mid X_1,\ldots,X_n)=\frac{\theta\nu(\cdot)+\sumN R_{X_i}(\cdot)}{\theta+\sumN R_{X_i}(\mathbb{X})}\qquad\mbox{a.s.},
\end{equation}
where $R$ is a non null\footnote{From an application point of view, it may be interesting to see whether exchangeability supports $R_x(\mathbb{X})=0$ for all $x\in Z$ in some $Z\in\mathcal{X}$ such that $0<\nu(Z)<1$, and to characterize $R_x$ for $x\in Z^c$. Importantly, by slightly modifying the proof of Theorem \ref{result:general:balance}, we can show in the non-i.i.d. case that $R_x(\mathbb{X})$ is constant for $\nu$-a.e. $x\in Z^c$. See Remark \ref{result:general:balance_remark} for more information.} finite transition kernel on $\mathbb{X}$, called the \textit{replacement/reinforcement kernel} of the process, $\nu$ is a probability measure on $\mathbb{X}$, known as the \textit{base measure}, and $\theta>0$ is a positive constant. We will say that the MVPS is \textit{balanced} if there is, in addition, some $m>0$ such that $R_x(\mathbb{X})=m$ for $\nu$-a.e. $x$, which in the urn analogy means that we add the same total number of balls each time. Moreover, unlike \eqref{intro:predictive:mvps}, we decompose the initial urn composition into two parameters, $\theta$ and $\nu$, each one having a separate effect on the process (see, e.g., the construction in \eqref{example:model:polya_sequence:limit}).

Consider an exchangeable MVPS$(\theta,\nu,R)$ with directing random measure $\tilde{P}$. It follows from \eqref{model:exchangeable:convergence} and \eqref{model:predictive} that the reinforcement $R$ directly influences the form of $\tilde{P}$. On the other hand, by Theorem \ref{model:predictive_characterization}, $R$ itself must satisfy certain conditions that make it admissible under exchangeability. Note, however, that equation \eqref{model:exchangeable:predictive_perm} is always true for MVPSs. Hence, information about $R$ can only be retrieved from the invariance property of the two-step-ahead predictive distributions in \eqref{model:exchangeable:predictive}. In fact, \citep[Theorem 7]{berti2023} show in their study of kernel based Dirichlet sequences that, for a balanced MVPS to be exchangeable, it is sufficient that equation \eqref{model:exchangeable:predictive} holds for $n=0,1$. Remarkably, when the MVPS is unbalanced, then \eqref{model:exchangeable:predictive} for $n=2$ is also needed (see the proof of Theorem \ref{result:general:balance}). Using this information, we show that, when normalized, $R$ must be a r.c.d. for $\nu$ given some sub-$\sigma$-algebra $\mathcal{G}$ of $\mathcal{X}$,
\begin{equation}\label{model:representation}
\frac{R_x(\cdot)}{R_x(\mathbb{X})}=\nu(\cdot\mid\mathcal{G})(x)\qquad\mbox{for }\nu\mbox{-a.e. }x,
\end{equation}
and in the process we prove that exchangeable MVPSs are necessarily balanced unless i.i.d. It turns out that the representation \eqref{model:representation} has major consequences for the shape and distribution of $\tilde{P}$, including the fact that the convergence in \eqref{model:exchangeable:convergence} is in total variation. We give one important example of an exchangeable MVPS next.

\begin{example}[\Polya sequence]\label{example:model:polya_sequence}
Let $(X_n)_{n\geq1}$ be an MVPS with reinforcement kernel
\begin{equation}\label{example:model:polya_sequence:reinforcement}
R_x=\delta_x\qquad\mbox{for }x\in\mathbb{X},
\end{equation}
i.e., we only reinforce the observed color with one additional ball. This process, also known as a \textit{\Polya sequence}, was first studied by Blackwell and MacQueen \cite{blackwell1973} as an extension of the \Polya urn scheme \eqref{intro:predictive:polya_urn} to general Polish spaces $\mathbb{X}$. In \cite{blackwell1973}, the authors prove that $(X_n)_{n\geq1}$ is exchangeable and its directing random measure $\tilde{P}$ has a \textit{Dirichlet process} distribution with parameters $(\theta,\nu)$, denoted $\DP(\theta,\nu)$. We recall that $\tilde{P}\sim\DP(\theta,\nu)$ if, for every finite partition $B_1,\ldots,B_n\in\mathcal{X}$ of $\mathbb{X}$, the random vector $(\tilde{P}(B_1),\ldots,\tilde{P}(B_n))$ has a Dirichlet distribution with parameters $(\theta\nu(B_1),\ldots,\theta\nu(B_n))$,
\[\bigl(\tilde{P}(B_1),\ldots,\tilde{P}(B_n)\bigr)\sim\Dir\bigl(\theta\nu(B_1),\ldots,\theta\nu(B_n)\bigr).\]
An equivalent statement (see, e.g., \cite[][p.112]{lijoi2010}) is that $\tilde{P}$ is equal in law to
\begin{equation}\label{example:model:polya_sequence:limit}
\tilde{P}(\cdot)\overset{w}{=}\sum_{j=1}^\infty V_j\delta_{Z_j}(\cdot),
\end{equation}
where $V_j=W_j\prod_{i=1}^{j-1}(1-W_i)$ and $(W_j)_{j\geq1}$ are i.i.d. with $W_1\sim\Be(1,\theta)$, and $(Z_j)_{j\geq1}$ are i.i.d.$(\nu)$, independent of $(V_j)_{j\geq1}$.

In this case, the reinforcement \eqref{example:model:polya_sequence:reinforcement} has the anticipated structure \eqref{model:representation}, since
\[R_x(\cdot)=\nu(\cdot\mid\mathcal{X})(x)\qquad\mbox{for }x\in\mathbb{X}.\]
\end{example}

\section{Main results}\label{section:results}

\subsection{General case}\label{section:results:general}

We begin this section with a simple characterization result for i.i.d. MVPSs.

\begin{proposition}\label{result:i.i.d}
An MVPS$(\theta,\nu,R)$ is i.i.d. if and only if, for $\nu$-a.e. $x$,
\begin{equation}\label{result:i.i.d.:equation:form}
\frac{R_x(\cdot)}{R_x(\mathbb{X})}=\nu(\cdot).
\end{equation}
\end{proposition}

The next theorem states that MVPSs that are exchangeable, but not i.i.d. are necessarily balanced. This fact is essential since it implies that the predictive distributions \eqref{model:predictive} are a linear combination of reinforcement kernels and leads via \eqref{model:exchangeable:predictive} to certain identities for $R$.

\begin{theorem}\label{result:general:balance}
An unbalanced exchangeable MVPS is necessarily i.i.d.
\end{theorem}

\begin{remark}\label{result:general:balance:remark}
If an MVPS$(\theta,\nu,R)$ is balanced, then $R_x(\mathbb{X})=m$ for $\nu$-a.e. $x$ and some constant $m>0$. It then follows from the particular form of the predictive distributions \eqref{model:predictive} that we can equivalently reformulate the process as an MVPS with parameters $(\tilde{\theta},\nu,\tilde{R})$, where $\tilde{\theta}=\frac{\theta}{m}$ and $\tilde{R}_x=\frac{R_x}{m}$ with $\tilde{R}_x(\mathbb{X})=1$, for $\nu$-a.e. $x$. On the other hand, it is not hard to see that any i.i.d. MVPS$(\theta,\nu,R)$ with unbalanced $R$ is also an i.i.d. MVPS with parameters $(\theta,\nu,\nu)$. Thus, when we consider i.i.d. MVPSs below, we will implicitly assume that they are balanced. The above remarks are combined in the next corollary.
\end{remark}

\begin{corollary}\label{result:general:balance_probability}
An exchangeable MVPS$(\theta,\nu,R)$ is also an exchangeable MVPS$(\tilde{\theta},\nu,\tilde{R})$ for some probability kernel $\tilde{R}$ and constant $\tilde{\theta}$.
\end{corollary}

We now consider the issue of representing $R$ as in \eqref{model:representation}. For completeness, we first state the converse result, which is found in \cite{berti2023}.

\begin{theorem}[Theorem 7 in \cite{berti2023}]\label{result:balance:conditional_distribution}
Any balanced MVPS$(\theta,\nu,R)$ whose reinforcement kernel $R$ is, up to normalization, a r.c.d. for $\nu$ given some sub-$\sigma$-algebra $\mathcal{G}$ of $\mathcal{X}$, is exchangeable.
\end{theorem}

\begin{remark}
Although Berti et al. \cite{berti2023} focus only on the balanced case, it follows from Theorem \ref{result:general:balance} that no unbalanced MVPS, having replacement kernel as in Theorem \ref{result:balance:conditional_distribution}, can be exchangeable unless $\mathcal{G}=\{\emptyset,\mathbb{X}\}$, in which case the observation sequence is i.i.d.
\end{remark}

Our main result studies the necessity of such a representation under exchangeability. In addition, it states that we may take $\mathcal{G}$ to be \textit{countably generated (c.g.) under $\nu$}, in the sense that there exists $C\in\mathcal{G}$ such that $\nu(C)=1$ and $\mathcal{G}\cap C$ is c.g. In fact (see \cite[p.649]{berti2007} and \cite{blackwell1975}), $\mathcal{G}$ is c.g. under $\nu$ if and only if there exists a regular version $\mu$ of $\nu(\cdot\mid\mathcal{G})$ which is a.e. \textit{proper}, that is, $\mu$ satisfies, for some $F\in\mathcal{G}$ such that $\nu(F)=1$, 
\[\mu_x(A)=\delta_x(A)\qquad\mbox{for all }A\in\mathcal{G}\mbox{ and }x\in F.\]

\begin{theorem}\label{result:general:main}
Let $(X_n)_{n\geq1}$ be an exchangeable MVPS$(\theta,\nu,R)$. Then there exists a sub-$\sigma$-algebra $\mathcal{G}$ of $\mathcal{X}$ such that $R$ is a r.c.d. for $\nu$ under $\mathcal{G}$,
\begin{equation}\label{result:general:equation:form}
\frac{R_x(\cdot)}{R_x(\mathbb{X})}=\nu(\cdot\mid\mathcal{G})(x)\qquad\mbox{for }\nu\mbox{-a.e. }x.
\end{equation}
Moreover, we can find a sub-$\sigma$-algebra $\mathcal{G}$ of $\mathcal{X}$ such that \eqref{result:general:equation:form} holds and $\mathcal{G}$ is c.g. under $\nu$.
\end{theorem}

\begin{remark}
If $(X_n)_{n\geq1}$ is an MVPS with reinforcement $R(\cdot)=\nu(\cdot\mid\mathcal{G})$ and thus exchangeable by Theorem \ref{result:balance:conditional_distribution}, then there exists by Theorem \ref{result:general:main} a sub-$\sigma$-algebra $\mathcal{G}'$ which is c.g. under $\nu$ and $\nu(\cdot\mid\mathcal{G})=\nu(\cdot\mid\mathcal{G}')$.
\end{remark}

A major consequence of Theorems \ref{result:general:balance} and \ref{result:general:main} is that one can now characterize the directing random measure $\tilde{P}$ of any exchangeable MVPS using results developed in \cite{berti2023} under the stronger assumptions that $R$ is balanced and of the form \eqref{result:general:equation:form}. In that case, Berti et al. \citep[][Theorem 10]{berti2023} show that the predictive distributions \eqref{model:predictive} converge almost surely in total variation to $\tilde{P}$. By virtue of Theorems \ref{result:general:balance} and \ref{result:general:main}, the following result, which is based on Theorems 9 and 10 in \citep{berti2023}, is now true for the entire class of exchangeable MVPS.

\begin{theorem}\label{result:general:limit}
Let $(X_n)_{n\geq1}$ be an exchangeable MVPS$(\theta,\nu,R)$ with directing random measure $\tilde{P}$. Then, as $n\rightarrow\infty$,
\[\sup_{A\in\mathcal{X}}\bigl|\mathbb{P}(X_{n+1}\in A|X_1,\ldots,X_n)-\tilde{P}(A)\bigr|\overset{a.s.}{\longrightarrow}0.\]
Moreover, $\tilde{P}$ is equal in law to
\[\tilde{P}(\cdot)\overset{w}{=}\sum_{j=1}^\infty V_j\frac{R_{Z_j}(\cdot)}{R_{Z_j}(\mathbb{X})},\]
where $(V_j)_{j\geq1}$ and $(Z_j)_{j\geq1}$ are as in \eqref{example:model:polya_sequence:limit}, using as parameters $(\frac{\theta}{m},\nu)$, and $m>0$ is some constant such that $m=R_x(\mathbb{X})$ for $\nu$-a.e. $x$.
\end{theorem}

\subsection{Absolutely continuous case}\label{section:results:abs_cont}

In this section, we consider MVPSs such that, for $\nu$-a.e. $x$, the replacement $R_x$ is \textit{absolutely continuous} with respect to the base measure $\nu$, denoted $R_x\ll\nu$. Special cases include $k$-color urn models and MVPSs with discrete $\nu$. The next theorem states that, under the assumption $R_x\ll\nu$, the $\sigma$-algebra $\mathcal{G}$ in \eqref{result:general:equation:form} is c.g. under $\nu$ w.r.t. a countable partition.

\begin{theorem}\label{result:abs_cont:main}
Let $(X_n)_{n\geq1}$ be an exchangeable MVPS$(\theta,\nu,R)$. Then $R_x\ll\nu$ for $\nu$-a.e. $x$ if and only if there exists a countable partition $D_1,D_2,\ldots$ of $\mathbb{X}$ in $\mathcal{X}$ such that
\begin{equation}\label{result:abs_cont:equation:form}
\frac{R_x(\cdot)}{R_x(\mathbb{X})}=\sum_k\nu(\cdot\mid D_k)\cdot\mathbbm{1}_{D_k}(x)\qquad\mbox{for }\nu\mbox{-a.e. }x.
\end{equation}
\end{theorem}

Note that the existence of the countable partition in \eqref{result:abs_cont:equation:form} does not depend on the enumerability of the color space $\mathbb{X}$ or on the form of the base measure $\nu$. The implications of Theorem \ref{result:abs_cont:main} are illustrated in the next Example \ref{example:abs_cont:k-color}, which describes $k$-color urn models with positive time-homogeneous reinforcement.

\begin{example}[$k$-color urn]\label{example:abs_cont:k-color}
When the space of colors $\mathbb{X}=\{x_1,\ldots,x_k\}$ is finite, then $\nu(\cdot)=\sum_{i=1}^kp_i\delta_{x_i}(\cdot)$, where $p_i\geq0$ represents the initial fraction of balls with color $x_i$ in the urn. Furthermore, it is customary to state $R$ in terms of a reinforcement matrix $[a_{ij}]_{1\leq i,j\leq k}$, where $a_{ij}\geq0$ denotes the number of balls of color $x_j$ that will be added to the urn after color $x_i$ has been observed.

If the observation process $(X_n)_{n\geq1}$ is exchangeable, then we have from \eqref{model:predictive} that
\[p_i=\mathbb{P}(X_2=x_i)=\sum_{j=1}^k\mathbb{P}(X_1=x_j)\mathbb{P}(X_2=x_i|X_1=x_j)=\sum_{j=1}^kp_j\frac{\theta p_i+a_{j,i}}{\theta+\sum_{l=1}^ka_{j,l}};\]
thus, if $p_i=0$ (i.e., there are initially no $x_i$ balls in the urn), then $a_{ji}=0$ for all $j:p_j>0$, so color $x_i$ never appears in the urn. The same argument can be repeated for any color that is not initially present in the urn, so we may assume $p_i>0$ for every $i$, without loss of generality. In that case, $R_x\ll\nu$ and $(X_n)_{n\geq 1}$ satisfies the conditions of Theorem \ref{result:abs_cont:main}.

Assuming $R_x(\mathbb{X})=1$, it follows from the particular form \eqref{result:abs_cont:equation:form} of $R$ that the reinforcement matrix of an exchangeable urn scheme with a finite number of colors has a block-diagonal design, i.e. $\mathbb{X}$ can be partitioned into subsets of colors $D_1,\ldots,D_l$, where a color reinforces and is reinforced only by the colors belonging to the same subset. Furthermore, the rows in each block are identical and equal to the conditional probability of $\nu$ given that a color from the same block has been observed. In other words, each $D_j$ can be viewed as representing different \emph{nuances} of the same color, so that the reinforcement $(i)$ will be identical for, say, light or dark red; and $(ii)$ will be equal to the initial number of balls for each nuance of red, normalized by the total number of red balls.

To illustrate the above points, let $\mathbb{X}=\{x_1,x_2,x_3\}$ and suppose that the urn has initial composition $(w_1,w_2,w_3)$, with $w_1,w_2,w_3>0$, so that $\nu(\cdot)=\sum_{i=1}^3\frac{w_i}{\bar{w}}\delta_{x_i}(\cdot)$ where $\bar{w}=\sum_{i=1}^3w_i$. Put $\bar{w}_i=\sum_{j\neq i}w_j$, $i=1,2,3$. Then any MVPS $(X_n)_{n\geq1}$ on $\mathbb{X}$ with initial composition $(w_1,w_2,w_3)$ is exchangeable if and only if its associated reinforcement matrix $R$ has one of the following forms, up to some multiplicative constants $m,m_1,m_2,m_3>0$,
\begingroup\allowdisplaybreaks
\begin{gather*}
\begin{blockarray}{cccc}
& x_1 & x_2 & x_3 \\
\begin{block}{c[ccc]}
x_1 & m_1\frac{w_1}{\bar{w}} & m_1\frac{w_2}{\bar{w}} & m_1\frac{w_3}{\bar{w}} \\
x_2 & m_2\frac{w_1}{\bar{w}} & m_2\frac{w_2}{\bar{w}} & m_2\frac{w_3}{\bar{w}} \\
x_3 & m_3\frac{w_1}{\bar{w}} & m_3\frac{w_2}{\bar{w}} & m_3\frac{w_3}{\bar{w}} \\
\end{block}\end{blockarray}
\qquad
\begin{blockarray}{ccc}
 & & \\
\begin{block}{[ccc]}
m\frac{w_1}{\bar{w}_3} & m\frac{w_2}{\bar{w}_3} & 0 \\
m\frac{w_1}{\bar{w}_3} & m\frac{w_2}{\bar{w}_3} & 0 \\
0 & 0 & m \\
\end{block}\end{blockarray}
\qquad
\begin{blockarray}{ccc}
 & & \\
\begin{block}{[ccc]}
m\frac{w_1}{\bar{w}_2} & 0 & m\frac{w_3}{\bar{w}_2} \\
0 & m & 0 \\
m\frac{w_1}{\bar{w}_2} & 0 & m\frac{w_3}{\bar{w}_2} \\
\end{block}\end{blockarray}\\
\begin{blockarray}{ccc}
\begin{block}{[ccc]}
m & 0 & 0 \\
0 & m\frac{w_2}{\bar{w}_1} & m\frac{w_3}{\bar{w}_1} \\
0 & m\frac{w_2}{\bar{w}_1} & m\frac{w_3}{\bar{w}_1} \\
\end{block}\end{blockarray}
\qquad
\begin{blockarray}{ccc}
\begin{block}{[ccc]}
m & 0 & 0 \\
0 & m & 0 \\
0 & 0 & m\\
\end{block}\end{blockarray}
\end{gather*}
\endgroup
Note that the first matrix corresponds to the case of an i.i.d. sequence of random variables with marginal distribution $\nu$, and the last one to the three-color \Polya urn model with a Dirichlet prior distribution with parameters $(\frac{w_1}{m},\frac{w_2}{m},\frac{w_3}{m})$.
\end{example}

\begin{example}[Discrete base measure]\label{example:abs_cont:discrete}
The previous example can be extended to MVPSs on more general spaces $\mathbb{X}$ when the base measure $\nu$ is discrete, which is true especially if $\mathbb{X}$ is countable. In this case, $\nu(\cdot)=\sum_{x\in\mathbb{X}_0} p(x)\delta_x(\cdot)$ for some countable subset $\mathbb{X}_0\subseteq\mathbb{X}$ and positive weights $p(x)>0$ such that $\sum_{x\in\mathbb{X}_0} p(x)=1$. If $(X_n)_{n\geq 1}$ is one such process, then it satisfies, by virtue of \eqref{model:predictive},
\[1=\sum_{x\in\mathbb{X}_0}p(x)=\mathbb{P}(X_1\in\mathbb{X}_0)=\mathbb{P}(X_1\in\mathbb{X},X_2\in\mathbb{X}_0)=\sum_{x\in\mathbb{X}_0}p(x)\frac{\theta+R_x(\mathbb{X}_0)}{\theta+R_x(\mathbb{X})};\]
thus, $R_x(\mathbb{X}_0^c)=0$, and so $R_x\ll\nu$ for every $x\in\mathbb{X}$. Theorem \ref{result:abs_cont:main} then implies the existence of a countable partition on $\mathbb{X}$ such that $\frac{R_x}{R_x(\mathbb{X})}$ is a r.c.d for $\nu$ given the $\sigma$-algebra generated by said partition.
\end{example}

The form of the replacement kernel in \eqref{result:abs_cont:equation:form} has further implications on the directing random measure $\tilde{P}$ of the exchangeable process $(X_n)_{n\geq1}$. In particular, it follows from \eqref{model:exchangeable:convergence} that, for every $A\in\mathcal{X}$, on a set of probability one,
\begingroup\allowdisplaybreaks
\begin{align*}
\tilde{P}(A)&=\limN\mathbb{P}(X_{n+1}\in A|X_1,\ldots,X_n)=\limN\sum_k\frac{\theta\nu(D_k)+\sumN R_{X_i}(\mathbb{X})\cdot\mathbbm{1}_{D_k}(X_i)}{\theta+\sumN R_{X_i}(\mathbb{X})}\nu(A|D_k)\\
&=\limN\sum_k\mathbb{P}(X_{n+1}\in D_k|X_1,\ldots,X_n)\nu(A|D_k)=\sum_k\tilde{P}(D_k)\nu(A|D_k);
\end{align*}
\endgroup
thus, as measures,
\begin{equation}\label{result:abs_cont:equation:a.s.limit}
\tilde{P}(\cdot)\overset{a.s.}{=}\sum_k\tilde{P}(D_k)\,\nu(\cdot\mid D_k),
\end{equation}
as $\mathcal{X}$ is countably generated. The next theorem is a direct consequence of this fact and the distribution results in Theorem \ref{result:general:limit} (see also Example 15 in \cite{berti2023} and Section 4.1.2 in \cite{thacker2022}).

\begin{theorem}\label{result:abs_cont:limit}
Let $(X_n)_{n\geq1}$ be an exchangeable MVPS$(\theta,\nu,R)$ with directing random measure $\tilde{P}$. If $R_x\ll\nu$ for $\nu$-a.e. $x$, then $\tilde{P}$ has the form \eqref{result:abs_cont:equation:a.s.limit} and, as $n\rightarrow\infty$, 
\[\sup_{A\in\mathcal{X}}\bigl|\mathbb{P}(X_{n+1}\in A|X_1,\ldots X_n)-\tilde{P}(A)\bigr|\overset{a.s.}{\longrightarrow}0.\]
Moreover, for every choice of indices $k_1,\ldots,k_j$,
\[\bigl(\tilde{P}(D_{k_1}),\ldots,\tilde{P}(D_{k_j}),\tilde{P}(\cap_{i=1}^jD_{k_i}^c)\bigr)\sim\Dir\Bigl(\frac{\theta}{m}\nu(D_{k_1}),\ldots,\frac{\theta}{m}\nu(D_{k_j}),\frac{\theta}{m}\nu(\cap_{i=1}^jD_{k_i}^c)\Bigr),\]
using the notation in Example \ref{example:model:polya_sequence}, where $m>0$ is such that $R_x(\mathbb{X})=m$ for $\nu$-a.e. $x$.
\end{theorem}

The directing random measure \eqref{result:abs_cont:equation:a.s.limit} implies that the observations $X_1,X_2,\ldots$ will form clusters on the level of the sets $D_1,D_2,\ldots$. Let $\tilde{P}^*(\cdot)=\sum_k\tilde{P}(D_k)\delta_k(\cdot)$ and $\nu^*(\cdot)=\sum_k\nu(D_k)\delta_k(\cdot)$. By Theorem \ref{result:abs_cont:limit}, any exchangeable MVPS with absolutely continuous reinforcement is a Dirichlet process mixture model
\begin{eqnarray}
X_n\mid\xi_n,\tilde{P}^*&\overset{ind.}{\sim}&\nu(\cdot\mid D_{\xi_n}) \nonumber\\
\xi_n\mid\tilde{P}^*&\overset{i.i.d.}{\sim}&\tilde{P}^* \nonumber\\
\tilde{P}^*&\sim&\mbox{DP}\Bigl(\frac{\theta}{m},\nu^*\Bigr) \nonumber
\end{eqnarray}
using the notation in Example \ref{example:model:polya_sequence}, where $(\xi_n)_{n\geq1}$ is a sequence of exchangeable random labels such that $\xi_n=k$ if and only if $X_n\in D_k$. We refer to \citep[][Section 3.4]{lijoi2010} for more details on Dirichlet process mixtures.

\subsection{Singular case}\label{section:results:singular}

In general, we have the decomposition of the reinforcement
\[R_x=R_x^\perp+R_x^a,\]
for some finite measures $R_x^\perp$ and $R_x^a$ on $\mathbb{X}$ such that $R_x^\perp$ and $\nu$ are mutually singular, denoted by $R_x^\perp\perp\nu$, and $R_x^a\ll\nu$. The next theorem shows that, for fixed $x$, the replacement $R_x$ of any exchangeable MVPS is either mutually singular or absolutely continuous with respect to $\nu$. In addition, the support of any mutually singular component, $R_x^\perp$, does not intersect the support of any absolutely continuous one, $R_y^a$, implying a complete separation of the two regimes.

\begin{theorem}\label{result:decomposition:main}
Let $(X_n)_{n\geq1}$ be an exchangeable MVPS$(\theta,\nu,R)$. Then there exists a set $\mathcal{S}\in\mathcal{X}$ such that
\[R_x=R_x^\perp\quad\mbox{and}\quad R_x^\perp(\mathcal{S}^c)=0,\qquad\mbox{for }\nu\mbox{-a.e. }x\mbox{ in }\mathcal{S},\]
and
\[R_x=R_x^a\quad\mbox{and}\quad R_x^a(\mathcal{S})=0,\qquad\mbox{for }\nu\mbox{-a.e. }x\mbox{ in }\mathcal{S}^c.\]
\end{theorem}

\begin{remark}
It follows from Theorem \ref{result:decomposition:main} that whenever $(X_n)_{n\geq1}$ is \textit{not} i.i.d., then for $\nu$-a.e. $x$, the support of $R_x$, say $S_x\in\mathcal{X}$ such that $R_x(S_x)=1$, has a $\nu$-measure strictly smaller than one, $\nu(S_x)<1$. Indeed, this is obvious if we have an $\mathcal{S}\in\mathcal{X}$ as in Theorem \ref{result:decomposition:main} such that $0<\nu(\mathcal{S})<1$. If instead $\nu(\mathcal{S})=0$, then $R_x\ll\nu$ for $\nu$-a.e. $x$ and $R_x$ has a "block-diagonal" design by Theorem \ref{result:abs_cont:main} with at least two disjoint blocks. On the other hand, if $\nu(\mathcal{S})=1$, then $R_x\perp\nu$ and there exists $S_x\in\mathcal{X}$ such that $R_x(S_x)=1$ and $\nu(S_x)=0$, for $\nu$-a.e. $x$.
\end{remark}


The following simple example demonstrates one such case.

\begin{example}
Let $\nu$ be a diffuse probability measure, i.e. $\nu(\{x\})=0$, for $x\in\mathbb{X}$. Fix $\theta=1$. Take $\mathcal{S}\in\mathcal{X}$ such that $0<\nu(\mathcal{S})<1$. Let us define, for $x\in\mathbb{X}$,
\[R_x(\cdot):=\biggl\{\begin{array}{cc} \delta_x(\cdot), & x\in\mathcal{S}; \\ \nu(\cdot\mid\mathcal{S}^c), & x\in\mathcal{S}^c.\end{array}\]
Then $R_x\perp\nu$ and $R_x(\mathcal{S}^c)=0$, for $x\in\mathcal{S}$, and $R_x\ll\nu$ and $R_x(\mathcal{S})=0$, for $x\in\mathcal{S}^c$.

Let $A,B\in\mathcal{X}$ and $x\in\mathbb{X}$. It is an easy exercise to check that
\[\int_AR_x(B)\nu(dx)=\int_BR_x(A)\nu(dx)\qquad\mbox{and}\qquad\int_AR_y(B)R_x(dy)=\int_BR_y(A)R_x(dy).\]
Therefore, by \citep[][Theorem 7]{berti2023}, the balanced MVPS $(X_n)_{n\geq1}$ with parameters $(\theta,\nu,R)$ is exchangeable.

Define
\[\mathcal{G}:=\bigl\{A\in\mathcal{X}:\nu(A|\mathcal{S}^c)=0\mbox{ or }\nu(A|\mathcal{S}^c)=1\bigr\}.\]
Then $\mathcal{G}$ is a $\sigma$-algebra.

Let $A\in\mathcal{G}$ and $B\in\mathcal{X}$. If $\nu(A|\mathcal{S}^c)=0$, then $\nu(A\cap\mathcal{S}^c)=0$, so $\nu(A\cap\mathcal{S}^c)\nu(B|\mathcal{S}^c)=\nu(A\cap B\cap\mathcal{S}^c)$; else, if $\nu(A|\mathcal{S}^c)=1$, then $\nu(A\cap\mathcal{S}^c)=\nu(\mathcal{S}^c)$ and $\nu(A^c\cap\mathcal{S}^c)=0$, from where $\nu(A^c\cap B\cap\mathcal{S}^c)=0$, and so $\nu(A\cap\mathcal{S}^c)\nu(B|\mathcal{S}^c)=\nu(B\cap\mathcal{S}^c)=\nu(A\cap B\cap\mathcal{S}^c)$. As a result,
\begingroup\allowdisplaybreaks
\begin{align*}
\int_AR_x(B)\nu(dx)&=\int_{A\cap\mathcal{S}}\delta_x(B)\nu(dx)+\int_{A\cap\mathcal{S}^c}\nu(B|\mathcal{S}^c)\nu(dx)\\
&=\nu(A\cap B\cap\mathcal{S})+\nu(A\cap\mathcal{S}^c)\nu(B|\mathcal{S}^c)=\nu(A\cap B\cap\mathcal{S})+\nu(A\cap B\cap\mathcal{S}^c)=\nu(A\cap B).
\end{align*}
\endgroup
On the other hand, for all $t\in[0,1]$ and $B\in\mathcal{X}$, the set $\{x\in\mathbb{X}:R_x(B)\leq t\}$ is equal to one of $B^c\cap\mathcal{S}$, $B^c\cup\mathcal{S}^c$, $\mathcal{S}$ or $\emptyset$, all in $\mathcal{G}$. Therefore, $x\mapsto R_x(B)$ is $\mathcal{G}$-measurable and
\[R_x(\cdot)=\nu(\cdot\mid\mathcal{G})(x)\qquad\mbox{for }\nu\mbox{-a.e. }x.\]
\end{example}

The decomposition implied in Theorem \ref{result:decomposition:main} can be used to analyze separately the singular and the absolutely continuous parts of $R$. Concentrating on $R_x^\perp$, let us assume that $R_x\perp\nu$ for all $x\in\mathbb{X}$. By Theorem \ref{result:general:balance}, such an MVPS is necessarily balanced, say $R_x(\mathbb{X})=1$ for $\nu$-a.e. $x$. Moreover, $\nu$ has to be diffuse; else, if $\nu(\{z\})>0$ for some $z\in\mathbb{X}$, then $R_x(\{z\})=0$, and using \eqref{model:predictive},
\[\nu(\{z\})=\int_\mathbb{X}\frac{\theta\nu(\{z\})+R_x(\{z\})}{\theta+R_x(\mathbb{X})}\nu(dx)=\frac{\theta}{\theta+1}\nu(\{z\}),\]
absurd. We show in Proposition \ref{result:singular:discrete:atoms} below that if $R_x$ is further discrete, then $x$ is an atom of $R_x$ for $\nu$-a.e. $x$. Examples include the \Polya sequence from Example \ref{example:model:polya_sequence} with diffuse $\nu$, and Example 2 in \cite{berti2023} where the authors consider an MVPS with $R_x=\frac{1}{2}(\delta_x+\delta_{-x})$. In addition, Proposition \ref{result:singular:discrete:atoms} states for discrete $R$ that either $R_x=R_y$ or $R_x\perp R_y$, which is reminiscent of the "block-diagonal" design in Section \ref{section:results:abs_cont}.

\begin{proposition}\label{result:singular:discrete:atoms}
Let $(X_n)_{n\geq1}$ be an exchangeable MVPS$(\theta,\nu,R)$. If $R_x\perp\nu$ and $R_x$ is discrete, then $R_x(\{x\})>0$ for $\nu$-a.e. $x$. Moreover,
\[\mbox{for }\nu\mbox{-a.e. }x\mbox{ and }y,\mbox{ either }R_x=R_y\mbox{ or }R_x\perp R_y.\]
\end{proposition}

\section{Proofs}\label{section:proofs}

\begin{proof}[Proof of Proposition \ref{result:i.i.d}] 
Let $(X_n)_{n\geq1}$ be an i.i.d. MVPS$(\theta,\nu,R)$. Take $A,B\in\mathcal{X}$. From \eqref{model:predictive},
\[\mathbb{P}(X_1\in A,X_2\in B)=\int_A\mathbb{P}(X_2\in B|X_1=x)\mathbb{P}(X_1\in dx)=\int_A\frac{\theta\nu(B)+R_x(B)}{\theta+R_x(\mathbb{X})}\nu(dx).\]
On the other hand,
\[\mathbb{P}(X_1\in A,X_2\in B)=\mathbb{P}(X_1\in A)\mathbb{P}(X_2\in B)=\nu(A)\nu(B)=\int_A\nu(B)\nu(dx).\]
Since $A$ is arbitrary, then $\frac{\theta\nu(B)+R_x(B)}{\theta+R_x(\mathbb{X})}=\nu(B)$ for $\nu$-a.e. $x$, and so $\frac{R_x(B)}{R_x(\mathbb{X})}=\nu(B)$ for $\nu$-a.e. $x$. As $\mathcal{X}$ is countably generated, we get, as measures, $\frac{R_x(dy)}{R_x(\mathbb{X})}=\nu(dy)$ for $\nu$-a.e. $x$.

Conversely, if $R$ is of the form \eqref{result:i.i.d.:equation:form}, then, for any $A\in\mathcal{X}$, 
\[\mathbb{P}(X_{n+1}\in A|X_1,\ldots,X_n)=\frac{\theta\nu(A)+\sumN R_{X_i}(\mathbb{X})\nu(A)}{\theta+\sumN R_{X_i}(\mathbb{X})}=\nu(A)\qquad\mbox{a.s.},\]
implying that $(X_n)_{n\geq1}$ is i.i.d. with marginal distribution $\nu$.
\end{proof}

\begin{proof}[Proof of Theorem \ref{result:general:balance}]
Let $(X_n)_{n\geq1}$ be an exchangeable MVPS$(\theta,\nu,R)$. Define
\[f(x):=R_x(\mathbb{X}),\quad\mbox{for }x\in\mathbb{X},\qquad\mbox{and}\qquad H:=\bigl\{(x,y)\in\mathbb{X}^2:f(x)=f(y)\bigr\},\]
and let $H_x$ be the $x$-section of $H$. It follows from exchangeability and the form of the predictive distributions \eqref{model:predictive} that, for every $A,B\in\mathcal{X}$,
\begingroup\allowdisplaybreaks
\begin{align}
\begin{aligned}\label{proof:general:balance:identity1_original}
\int_{(x,y)\in A\times B}\frac{\theta\nu(dy)+R_x(dy)}{\theta+f(x)}\nu(dx)&=\mathbb{P}(X_1\in A,X_2\in B)\\
&=\mathbb{P}(X_1\in B,X_2\in A)=\int_{(x,y)\in A\times B}\frac{\theta\nu(dx)+R_y(dx)}{\theta+f(y)}\nu(dy).
\end{aligned}
\end{align}
\endgroup 

We proceed by making some important observations from \eqref{proof:general:balance:identity1_original}. Let $C\in\mathcal{X}$ be such that $\nu(C)=0$. Applying \eqref{proof:general:balance:identity1_original} to $C\times\mathbb{X}$, we get $\nu(C)=\int_\mathbb{X}\frac{R_x(C)}{\theta+f(x)}\nu(dx)$, so $\frac{R_x(C)}{\theta+f(x)}=0$ for $\nu$-a.e. $x$, and hence
\begin{equation}\label{proof:general:balance:identity1_resul1}
\mbox{for any }C\in\mathcal{X},\mbox{ if }\nu(C)=0,\mbox{ then }R_x(C)=0\mbox{ for }\nu\mbox{-a.e. }x.
\end{equation}
On the other hand, \eqref{proof:general:balance:identity1_original} implies the equivalence of the measures
\[\frac{\theta\nu(dy)+R_x(dy)}{\theta+f(x)}\nu(dx)=\frac{\theta\nu(dx)+R_y(dx)}{\theta+f(y)}\nu(dy);\]
thus, $\nu(dx)=\int_{y\in\mathbb{X}}\frac{\theta\nu(dx)+R_y(dx)}{\theta+f(y)}\nu(dy)$, and by rearranging the terms,
\begin{equation}\label{proof:general:balance:identity1_resul2}
\int_{y\in\mathbb{X}}\frac{R_y(dx)}{\theta+f(y)}\nu(dy)=\int_{y\in\mathbb{X}}\frac{f(y)}{\theta+f(y)}\nu(dy)\nu(dx).
\end{equation}
Moreover, the following two measures, defined as the indefinite integrals of $(\theta+f(x))(\theta+f(y))$ with respect to $\frac{\theta\nu(dy)+R_x(dy)}{\theta+f(x)}\nu(dx)$ and $\frac{\theta\nu(dx)+R_y(dx)}{\theta+f(y)}\nu(dy)$, are equal
\[(\theta+f(y))\bigl(\theta\nu(dy)+R_x(dy)\bigr)\nu(dx)=(\theta+f(x))\bigl(\theta\nu(dx)+R_y(dx)\bigr)\nu(dy),\]
which after simplification becomes
\begin{equation}\label{proof:general:balance:identity1}
\theta f(y)\nu(dx)\nu(dy)+\bigl(\theta+f(y)\bigr)R_x(dy)\nu(dx)=\theta f(x)\nu(dx)\nu(dy)+\bigl(\theta+f(x)\bigr)R_y(dx)\nu(dy).
\end{equation}
In particular, restricting the above measures to $H$ where $f(x)=f(y)$, we obtain
\begin{equation}\label{proof:general:balance:condition1a}
\mathbbm{1}_H(x,y)R_x(dy)\nu(dx)=\mathbbm{1}_H(x,y)R_y(dx)\nu(dy).
\end{equation}

We now consider the implications that the invariance of the two-step-ahead predictive distributions has on $R$. Let $A,B\in\mathcal{X}$. Since $X_1\sim\nu$, it follows from \eqref{model:exchangeable:predictive} in Theorem \ref{model:predictive} that, for $\nu$-a.e. $x$,
\begingroup\allowdisplaybreaks
\begin{align*}
&\int_{(y,z)\in A\times B}\frac{\theta\nu(dz)+R_x(dz)+R_y(dz)}{\theta+f(x)+f(y)}\frac{\theta\nu(dy)+R_x(dy)}{\theta+f(x)}=\mathbb{P}(X_2\in A,X_3\in B|X_1=x)\\
&\hspace{5em}=\mathbb{P}(X_2\in B,X_3\in A|X_1=x)=\int_{(y,z)\in A\times B}\frac{\theta\nu(dy)+R_x(dy)+R_z(dy)}{\theta+f(x)+f(z)}\frac{\theta\nu(dz)+R_x(dz)}{\theta+f(x)}.
\end{align*}
\endgroup
Since $\mathcal{X}$ is countably generated, we may combine the non-null sets and get, for $\nu$-a.e. $x$, the equivalence of the measures
\[\frac{\theta\nu(dz)+R_x(dz)+R_y(dz)}{(\theta+f(x)+f(y))(\theta+f(x))}\bigl(\theta\nu(dy)+R_x(dy)\bigr)=\frac{\theta\nu(dy)+R_x(dy)+R_z(dy)}{(\theta+f(x)+f(z))(\theta+f(x))}\bigl(\theta\nu(dz)+R_x(dz)\bigr).\]
Arguing as in \eqref{proof:general:balance:identity1}, we get for $\nu$-a.e. $x$
\begingroup\allowdisplaybreaks
\begin{align}
\begin{aligned}\label{proof:general:balance:identity2_original}
\bigl(\theta+f(x)+f(z)\bigr)&\bigl(\theta\nu(dz)+R_x(dz)+R_y(dz)\bigr)\bigl(\theta\nu(dy)+R_x(dy)\bigr)\\
&=\bigl(\theta+f(x)+f(y)\bigr)\bigl(\theta\nu(dy)+R_x(dy)+R_z(dy)\bigr)\bigl(\theta\nu(dz)+R_x(dz)\bigr),
\end{aligned}
\end{align}
\endgroup
which after simplification\footnote{Note that the measures in \eqref{proof:general:balance:identity2_original}, and in the sequel, need not be finite, so extra care should be taken when cancelling common terms. The measures are, however, $\sigma$-finite with respect to the product of the sets $F_n:=\{x\in\mathbb{X}:n^{-1}\leq f(x)\leq n\}$, $n\geq1$; hence, if one works with the measures in \eqref{proof:general:balance:identity2_original} restricted to $(x,y,z)\in F_n\times F_n\times F_n$, all quantities from there on will be finite. The general result then follows from monotone convergence, letting $\mathbbm{1}_{F_n}\uparrow 1$ as $n\rightarrow\infty$. For exposition reasons, we will omit this procedure and work as if all quantities were finite to begin with.
} becomes
\begingroup\allowdisplaybreaks
\begin{align*}
&\begin{aligned}
\theta^2\Bigl\{&f(z)\nu(dy)\nu(dz)+R_y(dz)\nu(dy)\Bigr\}\\
&+\theta\Bigl\{f(z)R_x(dy)\nu(dz)+f(z)R_x(dz)\nu(dy)+\bigl(f(x)+f(z)\bigr)R_y(dz)\nu(dy)+R_y(dz)R_x(dy)\Bigr\}\\
&+\Bigl\{f(z)R_x(dy)R_x(dz)+\bigl(f(x)+f(z)\bigr)R_y(dz)R_x(dy)\Bigr\}
\end{aligned}\\
&\begin{aligned}
=\theta^2\Bigl\{&f(y)\nu(dy)\nu(dz)+R_z(dy)\nu(dz)\Bigr\}\\
&+\theta\Bigl\{f(y)R_x(dz)\nu(dy)+f(y)R_x(dy)\nu(dz)+\bigl(f(x)+f(y)\bigr)R_z(dy)\nu(dz)+R_z(dy)R_x(dz)\Bigr\}\\
&+\Bigl\{f(y)R_x(dy)R_x(dz)+\bigl(f(x)+f(y)\bigr)R_z(dy)R_x(dz)\Bigr\}.
\end{aligned}
\end{align*}
\endgroup
Using \eqref{proof:general:balance:identity1} w.r.t. $(y,z)$, we may collect $\theta f(z)R_y(dz)\nu(dy)$ and the two terms with $\theta^2$ on the left-hand side and cancel them with the sum of $\theta f(y)R_z(dy)\nu(dy)$ and the two terms with $\theta^2$ on the right-hand side to obtain
\begingroup\allowdisplaybreaks
\begin{align}\label{proof:general:balance:identity2}
\begin{aligned}
&\begin{aligned}\theta f(z)R_x(dy)\nu(dz)&+\theta f(z)R_x(dz)\nu(dy)+\theta f(x)R_y(dz)\nu(dy)\\
&+f(z)R_x(dy)R_x(dz)+\bigl(\theta+f(x)+f(z)\bigr)R_y(dz)R_x(dy)\end{aligned}\\
&\hspace{2em}\begin{aligned}=\theta f(y)R_x(dz)\nu(dy)&+\theta f(y)R_x(dy)\nu(dz)+\theta f(x)R_z(dy)\nu(dz)\\
&+f(y)R_x(dy)R_x(dz)+\bigl(\theta+f(x)+f(y)\bigr)R_z(dy)R_x(dz);\end{aligned}
\end{aligned}
\end{align}
\endgroup
thus, restricting the above measures to $H$ where $f(y)=f(z)$,
\begingroup\allowdisplaybreaks
\begin{align*}
\mathbbm{1}_{H}(y,z)\theta f(x)R_y(dz)\nu(dy)&+\mathbbm{1}_{H}(y,z)\bigl(\theta+f(x)+f(y)\bigr)R_y(dz)R_x(dy)\\
&=\mathbbm{1}_{H}(y,z)\theta f(x)R_z(dy)\nu(dz)+\mathbbm{1}_{H}(y,z)\bigl(\theta+f(x)+f(y)\bigr)R_z(dy)R_x(dz).
\end{align*}
\endgroup
By \eqref{proof:general:balance:condition1a}, we have $\mathbbm{1}_{H}(y,z)R_y(dz)\nu(dy)=\mathbbm{1}_{H}(y,z)R_z(dy)\nu(dz)$, and so, for $\nu$-a.e. $x$,
\[\mathbbm{1}_{H}(y,z)\bigl(\theta+f(x)+f(y)\bigr)R_y(dz)R_x(dy)=\mathbbm{1}_{H}(y,z)\bigl(\theta+f(x)+f(y)\bigr)R_z(dy)R_x(dz),\]
which after ``dividing'' by $\bigl(\theta+f(x)+f(y)\bigr)$, i.e. using again the notion of indefinite integrals, becomes 
\begin{equation}\label{proof:general:balance:condition2a}
\mathbbm{1}_{H}(y,z)R_y(dz)R_x(dy)=\mathbbm{1}_{H}(y,z)R_z(dy)R_x(dz).
\end{equation}

Next, considering equation \eqref{model:exchangeable:predictive} for $n=2$, we get that, as measures, 
\begin{equation}\label{proof:general:balance:identity3_original}
\mathbb{P}(X_3\in dz,X_4\in du|X_1=x,X_2=y)=\mathbb{P}(X_3\in du,X_4\in dz|X_1=x,X_2=y),
\end{equation}
for $\mathbb{P}_{(X_1,X_2)}$-a.e. $(x,y)$ in $\mathbb{X}^2$. It is an easy observation from \eqref{proof:general:balance:identity1_original} that $(\nu\times\nu)\ll\mathbb{P}_{(X_1,X_2)}$; thus, \eqref{proof:general:balance:identity3_original} holds true for $(\nu\times\nu)$-a.e. $(x,y)$ in $\mathbb{X}^2$. Equation \eqref{proof:general:balance:identity3_original} implies the equivalence of the measures
\begingroup\allowdisplaybreaks
\begin{align*}
&\frac{\theta\nu(du)+R_x(du)+R_y(du)+R_z(du)}{\theta+f(x)+f(y)+f(z)}\frac{\theta\nu(dz)+R_x(dz)+R_y(dz)}{\theta+f(x)+f(y)}\\
&\hspace{8em}=\frac{\theta\nu(dz)+R_x(dz)+R_y(dz)+R_u(dz)}{\theta+f(x)+f(y)+f(u)}\frac{\theta\nu(du)+R_x(du)+R_y(du)}{\theta+f(x)+f(y)},
\end{align*}
\endgroup
so arguing as in \eqref{proof:general:balance:identity2_original},
\begingroup\allowdisplaybreaks
\begin{align*}
\bigl(\theta+f(x)&+f(y)+f(u)\bigr)\bigl(\theta\nu(du)+R_x(du)+R_y(du)+R_z(du)\bigr)\bigl(\theta\nu(dz)+R_x(dz)+R_y(dz)\bigr)\\
&=\bigl(\theta+f(x)+f(y)+f(z)\bigr)\bigl(\theta\nu(dz)+R_x(dz)+R_y(dz)+R_u(dz)\bigr)\bigl(\theta\nu(du)+R_x(du)+R_y(du)\bigr).
\end{align*}
\endgroup
After some simple algebra, we obtain
\begingroup\allowdisplaybreaks
\begin{align*}
&\begin{aligned}
\theta^2\bigl\{f(u)&\nu(dz)\nu(du)+R_z(du)\nu(dz)\bigr\}\\
&+\theta\bigl\{f(u)R_x(du)\nu(dz)+f(u)R_y(du)\nu(dz)+\bigl(f(x)+f(y)+f(u)\bigr)R_z(du)\nu(dz)\\
&\hspace{3em}+f(u)R_x(dz)\nu(du)+f(u)R_y(dz)\nu(du)+R_z(du)R_x(dz)+R_z(du)R_y(dz)\bigr\}\\
&+\bigl\{f(u)R_x(dz)R_x(du)+f(u)R_y(du)R_x(dz)+\bigl(f(x)+f(y)+f(u)\bigr)R_z(du)R_x(dz)\\
&\hspace{3em}+f(u)R_x(du)R_y(dz)+f(u)R_y(dz)R_y(du)+\bigl(f(x)+f(y)+f(u)\bigr)R_z(du)R_y(dz)\bigr\}\end{aligned}\\
&\begin{aligned}
=\theta^2\bigl\{f(z)&\nu(dz)\nu(du)+R_u(dz)\nu(du)\bigr\}\\
&+\theta\bigl\{f(z)R_x(dz)\nu(du)+f(z)R_y(dz)\nu(du)+\bigl(f(x)+f(y)+f(z)\bigr)R_u(dz)\nu(du)\\
&\hspace{3em}+f(z)R_x(du)\nu(dz)+f(z)R_y(du)\nu(dz)+R_u(dz)R_x(du)+R_u(dz)R_y(du)\bigr\}\\
&+\bigl\{f(z)R_x(dz)R_x(du)+f(z)R_y(dz)R_x(du)+\bigl(f(x)+f(y)+f(z)\bigr)R_u(dz)R_x(du)\\
&\hspace{3em}+f(z)R_x(dz)R_y(du)+f(z)R_y(dz)R_y(du)+\bigl(f(x)+f(y)+f(z)\bigr)R_u(dz)R_y(du)\bigr\},\end{aligned}
\end{align*}
\endgroup
which after further rearrangement becomes
\begingroup\allowdisplaybreaks
\begin{align*}
&\begin{aligned}
\theta\bigl\{\theta f(u)&\nu(dz)\nu(du)+(\theta+f(u))R_z(du)\nu(dz)\bigr\}\\
&+\bigl\{\theta f(u)R_x(dz)\nu(du)+\theta f(u)R_x(du)\nu(dz)+\theta f(x)R_z(du)\nu(dz)\\
&\hspace{8em}+\bigl(\theta+f(x)+f(u)\bigr)R_z(du)R_x(dz)+f(u)R_x(dz)R_x(du)\bigr\}\\
&+\bigr\{\theta f(u)R_y(dz)\nu(du)+\theta f(u)R_y(du)\nu(dz)+\theta f(y)R_z(du)\nu(dz)\\
&\hspace{8em}+\bigl(\theta+f(y)+f(u)\bigr)R_z(du)R_y(dz)+f(u)R_y(dz)R_y(du)\bigr\}\\
&+\bigl\{f(x)R_z(du)R_y(dz)+f(u)R_x(dz)R_y(du)+f(y)R_z(du)R_x(dz)+f(u)R_x(du)R_y(dz)\bigr\}
\end{aligned}\\
&\begin{aligned}
=\theta\bigl\{\theta f(z)&\nu(dz)\nu(du)+(\theta+f(z))R_u(dz)\nu(du)\bigr\}\\
&+\bigl\{\theta f(z)R_x(du)\nu(dz)+\theta f(z)R_x(dz)\nu(du)+\theta f(x)R_u(dz)\nu(du)\\
&\hspace{8em}+\bigl(\theta+f(x)+f(z)\bigr)R_u(dz)R_x(du)+f(z)R_x(dz)R_x(du)\bigr\}\\
&+\bigl\{\theta f(z)R_y(du)\nu(dz)+\theta f(z)R_y(dz)\nu(du)+\theta f(y)R_u(dz)\nu(du)\\
&\hspace{8em}+\bigl(\theta+f(y)+f(z)\bigr)R_u(dz)R_y(du)+f(z)R_y(dz)R_y(du)\bigr\}\\
&+\bigl\{f(x)R_u(dz)R_y(du)+f(z)R_x(du)R_y(dz)+f(y)R_u(dz)R_x(du)+f(z)R_x(dz)R_y(du)\bigr\}.\end{aligned}
\end{align*}
\endgroup
It follows from \eqref{proof:general:balance:identity1} w.r.t. $(z,u)$, \eqref{proof:general:balance:identity2} w.r.t. $(x,z,u)$, and \eqref{proof:general:balance:identity2} w.r.t. $(y,z,u)$ that we may cancel the terms in the first, second and third brackets, respectively, to get, for $(\nu\times\nu)$-a.e. $(x,y)$ in $\mathbb{X}^2$,
\begingroup\allowdisplaybreaks
\begin{align*}
f(x)R_z(du)R_y&(dz)+f(u)R_x(dz)R_y(du)+f(y)R_z(du)R_x(dz)+f(u)R_x(du)R_y(dz)\\
&=f(x)R_u(dz)R_y(du)+f(z)R_x(du)R_y(dz)+f(y)R_u(dz)R_x(du)+f(z)R_x(dz)R_y(du),
\end{align*}
\endgroup
from where\footnote{Again we refer to the footnote on p. 12 on how to deal with signed measures like those in \eqref{proof:general:balance:identity3} in order to avoid indeterminate forms.}
\begingroup\allowdisplaybreaks
\begin{align}
\begin{aligned}\label{proof:general:balance:identity3}
f(x)R_z(du)R_y(dz)&+f(y)R_z(du)R_x(dz)+\bigl(f(u)-f(z)\bigr)R_x(dz)R_y(du)+\bigl(f(u)-f(z)\bigr)R_x(du)R_y(dz)\\
&=f(x)R_y(du)R_u(dz)+f(y)R_x(du)R_u(dz).
\end{aligned}
\end{align}
\endgroup
Dividing both sides in \eqref{proof:general:balance:identity3} by $(\theta+f(x))(\theta+f(y))$ and integrating with respect to $\nu(dx)$ and $\nu(dy)$ yields
\begingroup\allowdisplaybreaks
\begin{align*}
&\begin{aligned}\int_\mathbb{X}\int_\mathbb{X}\frac{f(x)}{(\theta+f(x))(\theta+f(y))}&R_z(du)R_y(dz)\nu(dx)\nu(dy)\\
&+\int_\mathbb{X}\int_\mathbb{X}\frac{f(y)}{(\theta+f(x))(\theta+f(y))}R_z(du)R_x(dz)\nu(dx)\nu(dy)\\
&+\int_\mathbb{X}\int_\mathbb{X}\frac{f(u)-f(z)}{(\theta+f(x))(\theta+f(y))}R_x(dz)R_y(du)\nu(dx)\nu(dy)\\
&+\int_\mathbb{X}\int_\mathbb{X}\frac{f(u)-f(z)}{(\theta+f(x))(\theta+f(y))}R_x(du)R_y(dz)\nu(dx)\nu(dy)\end{aligned}\\
&\begin{aligned}=\int_\mathbb{X}\int_\mathbb{X}\frac{f(x)}{(\theta+f(x))(\theta+f(y))}&R_y(du)R_u(dz)\nu(dx)\nu(dy)\\
&+\int_\mathbb{X}\int_\mathbb{X}\frac{f(y)}{(\theta+f(x))(\theta+f(y))}R_x(du)R_u(dz)\nu(dx)\nu(dy).\end{aligned}
\end{align*}
\endgroup
Now, using the identity in \eqref{proof:general:balance:identity1_resul2} w.r.t. $\int_\mathbb{X}\frac{R_x(dz)}{\theta+f(x)}\nu(dx)$, $\int_\mathbb{X}\frac{R_x(du)}{\theta+f(x)}\nu(dx)$, $\int_\mathbb{X}\frac{R_y(dz)}{\theta+f(y)}\nu(dy)$ and $\int_\mathbb{X}\frac{R_y(du)}{\theta+f(y)}\nu(dy)$ on, respectively, the second and third terms on the left, the forth term on the left and the second on the right, the first and forth terms on the left, the third term on the left and the first on the right of the equation, leads to
\begingroup\allowdisplaybreaks
\begin{align*}
&\begin{aligned}2\int_\mathbb{X}\int_\mathbb{X}\frac{f(x)f(y)}{(\theta+f(x))(\theta+f(y))}&R_z(du)\nu(dx)\nu(dy)\nu(dz)\\
&+2\int_\mathbb{X}\int_\mathbb{X}\frac{f(x)f(y)\bigl(f(u)-f(z)\bigr)}{(\theta+f(x))(\theta+f(y))}\nu(dx)\nu(dy)\nu(dz)\nu(du)\end{aligned}\\
&=2\int_\mathbb{X}\int_\mathbb{X}\frac{f(x)f(y)}{(\theta+f(x))(\theta+f(y))}R_u(dz)\nu(dx)\nu(dy)\nu(du).
\end{align*}
\endgroup
Observe that the terms involving $z$ and $u$ can be pulled out of the integrals; thus, dividing both sides by $2\int_\mathbb{X}\int_\mathbb{X}\frac{f(x)f(y)}{(\theta+f(x))(\theta+f(y))}\nu(dx)\nu(dy)$, we get
\[R_z(du)\nu(dz)+\bigl(f(u)-f(z)\bigr)\nu(dz)\nu(du)=R_u(dz)\nu(du).\]
On the other hand, we have $R_u(dz)\nu(du)=\frac{\theta+f(u)}{\theta+f(z)}R_z(du)\nu(dz)+\theta\frac{f(u)-f(z)}{\theta+f(z)}\nu(dz)\nu(du)$ from \eqref{proof:general:balance:identity1}, so substituting above,
\[R_z(du)\nu(dz)+\bigl(f(u)-f(z)\bigr)\nu(dz)\nu(du)=\frac{\theta+f(u)}{\theta+f(z)}R_z(du)\nu(dz)+\theta\frac{f(u)-f(z)}{\theta+f(z)}\nu(dz)\nu(du),\]
from where
\[\bigl(f(u)-f(z)\bigr)R_z(du)\nu(dz)=\bigl(f(u)-f(z)\bigr)f(z)\nu(dz)\nu(du).\]
Therefore, working on $H^c$ where $f(z)\neq f(u)$, we obtain
\begin{equation}\label{proof:general:balance:equation:i.i.d.1}
\mathbbm{1}_{H^c}(z,u)R_z(du)\nu(dz)=\mathbbm{1}_{H^c}(z,u)f(z)\nu(dz)\nu(du).
\end{equation}

Going back to \eqref{proof:general:balance:identity2}, we now focus on the triplets $(x,y,z)$ in $\mathbb{X}^3$ for which $f(x)=f(y)$ and $f(x)\neq f(z)$, so taking the indefinite integral of the terms in \eqref{proof:general:balance:identity2} with respect to $\nu(dx)$, we get that, after some simple algebra,
\begingroup\allowdisplaybreaks
\begin{align}
\begin{aligned}\label{proof:general:balance:equation:i.i.d.2:prelim}
&\begin{aligned}\mathbbm{1}_H(x,y)\mathbbm{1}_{H^c}(x,z)\Bigl\{&\theta\bigl(f(z)-f(x)\bigr)R_x(dy)\nu(dz)\nu(dx)+\theta\bigl(f(z)-f(x)\bigr)R_x(dz)\nu(dy)\nu(dx)\\
&+\bigl(f(z)-f(x)\bigr)R_x(dy)R_x(dz)\nu(dx)+\theta f(x)R_y(dz)\nu(dy)\nu(dx)\\
&+\bigl(\theta+f(x)+f(z)\bigr)R_y(dz)R_x(dy)\nu(dx)\Bigr\}\end{aligned}\\
&=\mathbbm{1}_H(x,y)\mathbbm{1}_{H^c}(x,z)\Bigl\{\theta f(x)R_z(dy)\nu(dz)\nu(dx)+\bigl(\theta+2f(x)\bigr)R_z(dy)R_x(dz)\nu(dx)\Bigr\}.
\end{aligned}
\end{align}
\endgroup
On the set where $f(y)=f(x)\neq f(z)$, equation \eqref{proof:general:balance:equation:i.i.d.1} implies $(i)$ $R_x(dz)=f(x)\nu(dz)$ for $\nu$-a.e. $x$; $(ii)$ $R_y(dz)=f(y)\nu(dz)=f(x)\nu(dz)$ for $\nu$-a.e. $y$; and $(iii)$ $R_z(dy)=f(z)\nu(dy)$ for $\nu$-a.e. $z$. Moreover, by \eqref{proof:general:balance:identity1_resul1}, $(ii)$ implies for $\nu$-a.e. $x$ that $(iv)$ $R_y(dz)=f(x)\nu(dz)$ for $R_x$-a.e. $y$. Thus,
\begingroup\allowdisplaybreaks
\begin{gather*}
\mathbbm{1}_H(x,y)\mathbbm{1}_{H^c}(x,z)R_x(dz)\nu(dy)\nu(dx)=\mathbbm{1}_H(x,y)\mathbbm{1}_{H^c}(x,z)f(x)\nu(dx)\nu(dy)\nu(dz)\quad\mbox{by }(i),\\
\mathbbm{1}_H(x,y)\mathbbm{1}_{H^c}(x,z)R_x(dy)R_x(dz)\nu(dx)=\mathbbm{1}_H(x,y)\mathbbm{1}_{H^c}(x,z)f(x)R_x(dy)\nu(dz)\nu(dx)\quad\mbox{by }(i),\\
\mathbbm{1}_H(x,y)\mathbbm{1}_{H^c}(x,z)R_y(dz)\nu(dy)\nu(dx)=\mathbbm{1}_H(x,y)\mathbbm{1}_{H^c}(x,z)f(x)\nu(dx)\nu(dy)\nu(dz)\quad\mbox{by }(ii),\\
\mathbbm{1}_H(x,y)\mathbbm{1}_{H^c}(x,z)R_y(dz)R_x(dy)\nu(dx)=\mathbbm{1}_H(x,y)\mathbbm{1}_{H^c}(x,z)f(x)R_x(dy)\nu(dz)\nu(dx)\quad\mbox{by }(iv),\\
\mathbbm{1}_H(x,y)\mathbbm{1}_{H^c}(x,z)R_z(dy)\nu(dz)\nu(dx)=\mathbbm{1}_H(x,y)\mathbbm{1}_{H^c}(x,z)f(z)\nu(dx)\nu(dy)\nu(dz)\quad\mbox{by }(iii),
\end{gather*}
\endgroup
and, by $(i)$ and $(iii)$,
\[\mathbbm{1}_H(x,y)\mathbbm{1}_{H^c}(x,z)R_z(dy)R_x(dz)\nu(dx)=\mathbbm{1}_H(x,y)\mathbbm{1}_{H^c}(x,z)f(x)f(z)\nu(dx)\nu(dy)\nu(dz).\]
Substituting for the above terms in \eqref{proof:general:balance:equation:i.i.d.2:prelim}, we obtain
\begingroup\allowdisplaybreaks
\begin{align*}
&\begin{aligned}\mathbbm{1}_H(x,y)\mathbbm{1}_{H^c}(x,z)\Bigl\{&\theta\bigl(f(z)-f(x)\bigr)R_x(dy)\nu(dx)\nu(dz)+\theta f(x)\bigl(f(z)-f(x)\bigr)\nu(dx)\nu(dy)\nu(dz)\\
&+f(x)\bigl(f(z)-f(x)\bigr)R_x(dy)\nu(dx)\nu(dz)+\theta f(x)^2\nu(dx)\nu(dy)\nu(dz)\\
&+f(x)\bigl(\theta+f(x)+f(z)\bigr)R_x(dy)\nu(dx)\nu(dz)\Bigr\}\end{aligned}\\
&=\mathbbm{1}_H(x,y)\mathbbm{1}_{H^c}(x,z)\Bigl\{\theta f(x)f(z)\nu(dx)\nu(dy)\nu(dz)+f(x)f(z)\bigl(\theta+2f(x)\bigr)\nu(dx)\nu(dy)\nu(dz)\Bigr\},
\end{align*}
\endgroup
which after simplification turns into
\begingroup\allowdisplaybreaks
\begin{align}\label{proof:general:balance:equation:i.i.d.2_pre}
\begin{aligned}
\mathbbm{1}_H(x,y)\mathbbm{1}_{H^c}(x,z)\Bigl\{f(z)\bigl(\theta&+2f(x)\bigr)R_x(dy)\nu(dx)\nu(dz)\Bigr\}\\
&=\mathbbm{1}_H(x,y)\mathbbm{1}_{H^c}(x,z)\Bigl\{f(x)f(z)\bigl(\theta+2f(x)\bigr)\nu(dx)\nu(dy)\nu(dz)\Bigr\}.
\end{aligned}
\end{align}
\endgroup
Dividing by $f(z)\bigl(\theta+2f(x)\bigr)$ and integrating $z$ out leaves
\begin{equation}\label{proof:general:balance:equation:i.i.d.2}
\mathbbm{1}_H(x,y)\nu(H_x^c)R_x(dy)\nu(dx)=\mathbbm{1}_H(x,y)\nu(H_x^c)f(x)\nu(dx)\nu(dy).
\end{equation}

Let us define
\[D^*:=\{x\in\mathbb{X}:\nu(H_x^c)>0\}.\]
Suppose that $\nu(D^*)>0$. Take $y\in\mathbb{X}$. If $f(y)=f(x)$ for some $x\in D^*$, then $H_y^c=H_x^c$, and $\nu(H_y^c)=\nu(H_x^c)>0$; otherwise, if $f(y)\neq f(x)$ for all $x\in D^*$, then $D^*\subseteq H_y^c$ and $\nu(H_y^c)\geq\nu(D^*)>0$. In either case, $y\in D^*$, which implies that $\nu(D^*)=1$. As a result, we can divide both sides of \eqref{proof:general:balance:equation:i.i.d.2} by $\nu(H_x^c)$ to obtain
\[\mathbbm{1}_{H_x}(y)R_x(dy)\nu(dx)=\mathbbm{1}_{H_x}(y)f(x)\nu(dx)\nu(dy).\]
Summing this with \eqref{proof:general:balance:equation:i.i.d.1}, we get that, for $\nu$-a.e. $x$,
\[R_x(dy)=f(x)\nu(dy),\]
which implies from Proposition \ref{result:i.i.d} that $(X_n)_{n\geq1}$ is i.i.d.

If instead $\nu(D^*)=0$, then $\nu(H_x^c)=0$ for $\nu$-a.e. $x$. It follows for some fixed $x'$ with $\nu(H_{x'})=1$ that
\[f(x)=f(x')\qquad\mbox{for }\nu\mbox{-a.e. }x,\]
implying that the model is balanced. This concludes the proof of the theorem.
\end{proof}

\begin{remark}\label{result:general:balance_remark}
The proof of Theorem \ref{result:general:balance} largely remains the same even if we assume $R_x(\mathbb{X})=0$ for all $x$ in some $Z\in\mathcal{X}$ such that $0<\nu(Z)<1$. The only place of concern is in "dividing" \eqref{proof:general:balance:equation:i.i.d.2_pre} by $R_x(\mathbb{X})$, so by restricting \eqref{proof:general:balance:equation:i.i.d.2_pre} on $Z^c$ we get
\[\mathbbm{1}_H(x,y)\nu(H_x^c\cap Z^c)R_x(dy)\nu(dx)=\mathbbm{1}_H(x,y)\nu(H_x^c\cap Z^c)f(x)\nu(dx)\nu(dy).\]
Now, letting
\[D^*:=\{x\in\mathbb{X}:\nu(H_x^c\cap Z^c)>0\},\]
we obtain under $\nu(D^*\cap Z^c)>0$ that $(X_n)_{n\geq0}$ is i.i.d., using the same logic as above. If, on the other hand, $\nu(D^*\cap Z^c)=0$, then $\nu((D^*)^c)=\nu(Z^c)$, since $(D^*)^c\subseteq Z^c$, and so
\[(\nu\times\nu)\bigl(H^c\cap(Z^c\times Z^c)\bigr)=\int_{Z^c}\nu(H_x^c\cap Z^c)\nu(dx)=\int_{(D^*)^c}\nu(H_x^c\cap Z^c)\nu(dx)=0,\]
thus showing that $R_x(\mathbb{X})=m>0$ for $\nu$-a.e. $x$ in $Z^c$ and some $m>0$, when $(X_n)_{n\geq1}$ is not i.i.d.
\end{remark}

\begin{proof}[Proof of Theorem \ref{result:general:main}]
If $(X_n)_{n\geq 1}$ is i.i.d., then Proposition \ref{result:i.i.d} implies \eqref{result:general:equation:form} with $\mathcal{G}=\{\emptyset,\mathbb{X}\}$. Suppose that $(X_n)_{n\geq1}$ is exchangeable, but not i.i.d. It follows from Theorem \ref{result:general:balance} that the model is balanced, so we may assume, without loss of generality, that $R_x(\mathbb{X})=1$, for $x\in\mathbb{X}$; otherwise, we can normalize $\theta$ and $R$ (see Remark \ref{result:general:balance:remark}) and proceed from there. In this case, \eqref{proof:general:balance:condition1a} reduces to
\begin{equation}\label{proof:general:representation:identity1}
R_x(dy)\nu(dx)=R_y(dx)\nu(dy),
\end{equation}
and \eqref{proof:general:balance:condition2a} implies, taking the indefinite integral of the terms with respect to $\nu(dx)$, that
\begin{equation}\label{proof:general:representation:identity2}
R_y(dz)R_x(dy)\nu(dx)=R_z(dy)R_x(dz)\nu(dx).
\end{equation}
Applying \eqref{proof:general:representation:identity1} and \eqref{proof:general:representation:identity2} repeatedly, we obtain that
\begingroup\allowdisplaybreaks
\begin{align*}
R_y(dz)R_x(dy)\nu(dx)&=R_z(dy)R_x(dz)\nu(dx)=R_z(dy)R_z(dx)\nu(dz)\\
&=R_z(dx)R_y(dz)\nu(dy)=R_x(dz)R_y(dx)\nu(dy)=R_x(dz)R_x(dy)\nu(dx).
\end{align*}
\endgroup
As $\mathcal{X}$ is countably generated, there exists $C_0\in\mathcal{X}$ such that $\nu(C_0)=1$ and, for all $x\in C_0$,
\begin{equation}\label{proof:general:representation:identity3}
R_y(dz)R_x(dy)=R_x(dz)R_x(dy).
\end{equation}

Define $C_1:=\{x\in C_0:R_x(C_0)=1\}$ and $C_n:=\{x\in C_{n-1}:R_x(C_{n-1})=1\}$ for every $n\geq2$. As $R_x(\mathbb{X})=1$, we get from \eqref{proof:general:representation:identity1} that
\[1=\int_{C_0}R_x(\mathbb{X})\nu(dx)=\int_\mathbb{X}R_x(C_0)\nu(dx)=1-\int_{\mathbb{X}}(1-R_x(C_0))\nu(dx)=1-\int_{C_1^c}(1-R_x(C_0))\nu(dx).\]
But $R_x(C_0)<1$ for $x\in C_1^c$, so $\nu(C_1^c)=0$, otherwise the term on the right-hand side of the equation becomes strictly less than $1$. Proceeding by induction, we get $\nu(C_n)=1$ for all $n\geq1$; thus, letting $C^*=\bigcap_{n=1}^\infty C_n$, it holds $\nu(C^*)=1$ and $R_x(C_n)=1$, for all $x\in C^*$ and $n\geq1$. As a result, for every $x\in C^*$, we have that $R_x(C^*)=1$ and, since $C^*\subseteq C_0$, from \eqref{proof:general:representation:identity3},
\begin{equation}\label{proof:general:equality}
R_x(dz)=R_y(dz)\qquad\mbox{for }R_x\mbox{-a.e. }y.
\end{equation}

Let us define
\[\mathcal{G}:=\{A\in\mathcal{X}:R_x(A)=\delta_x(A)\mbox{ for all }x\in C^*\}.\]
It is not hard to see that $\mathcal{G}$ is a $\sigma$-algebra. Next, fix $t\in[0,1]$ and $B\in\mathcal{X}$. Denote
\[E_{t,B}=\{x\in\mathbb{X}:R_x(B)<t\}.\]
Let $x\in C^*$. If $x\in E_{t,B}$, then $R_y(B)=R_x(B)<t$ for $R_x$-a.e. $y$ by \eqref{proof:general:equality}, so $R_x(E_{t,B})=1$. If instead $x\in E_{t,B}^c$, then again $R_y(B)=R_x(B)\geq t$ for $R_x$-a.e. $y$ by \eqref{proof:general:equality}, and so $R_x(E_{t,B})=0$. As a result, $R_x(E_{t,B})=\delta_x(E_{t,B})$, which implies that $E_{t,B}\in\mathcal{G}$; thus, $x\mapsto R_x(B)$ is $\mathcal{G}$-measurable.

Let $A\in\mathcal{G}$ and $B\in\mathcal{X}$. It follows from \eqref{proof:general:representation:identity1} and $\nu(C^*)=1$ that
\[\int_AR_x(B)\nu(dx)=\int_BR_x(A)\nu(dx)=\int_{B\cap C^*}R_x(A)\nu(dx)=\int_B\delta_x(A)\nu(dx)=\nu(A\cap B).\]
Therefore, $R$ is a regular version of $\nu(\cdot\mid\mathcal{G})$. Moreover, by construction,
\[R_x(A)=\delta_x(A)\qquad\mbox{for all }A\in\mathcal{G}\mbox{ and }x\in C^*;\]
thus, $R$ is a.e. proper and $\mathcal{G}$ is c.g. under $\nu$ (see \cite[p.649]{berti2007} and \cite[Theorem 1]{blackwell1975}).
\end{proof}


\begin{proof}[Proof of Theorem \ref{result:abs_cont:main}]
If $(X_n)_{n\geq1}$ is i.i.d., then Proposition \ref{result:i.i.d} implies \eqref{result:abs_cont:equation:form} with respect to the partition $\{\mathbb{X},\emptyset\}$, and vice versa. Suppose that $(X_n)_{n\geq0}$ is exchangeable, but not i.i.d. Assume, w.l.o.g., that $R_x(\mathbb{X})=1$ for all $x\in\mathbb{X}$. If $R_x$ is of the form \eqref{result:abs_cont:equation:form}, then it is obvious that $R_x\ll\nu$ for $\nu$-a.e. $x$.

Conversely, suppose that $R_x\ll\nu$ for $\nu$-a.e. $x$. Then $R_x(dy)\nu(dx)\ll\nu(dx)\nu(dy)$, so there exists a jointly measurable function $r:\mathbb{X}^2\rightarrow\mathbb{R}_+$ such that
\[R_x(dy)\nu(dx)=r(x,y)\nu(dx)\nu(dy),\]
and, since $\mathcal{X}$ is countably generated, as measures,
\begin{equation}\label{proof:abs_cont:density}
R_x(dy)=r(x,y)\nu(dy)\qquad\mbox{for }\nu\mbox{-a.e. }x.
\end{equation}

Denote by
\[D:=\bigl\{(x,y)\in\mathbb{X}^2:r(x,y)>0\bigr\},\]
and, letting $D_x$ be the $x$-section of $D$ for $x\in\mathbb{X}$, define
\[G:=\{x\in\mathbb{X}:\nu(D_x)>0\}.\]
Then $R_x(D_x)=\int_{D_x}r(x,y)\nu(dy)=1$ for $\nu$-a.e. $x$. Moreover,
\[1=\int_\mathbb{X}R_x(D_x)\nu(dx)=\int_G\int_{D_x}r(x,y)\nu(dy)\nu(dx)+\int_{G^c}\int_{D_x}r(x,y)\nu(dy)\nu(dx)=\nu(G).\]
On the other hand, by \eqref{proof:general:representation:identity1}, \eqref{proof:general:representation:identity2} and \eqref{proof:general:representation:identity3}, we have the following identities
\begingroup\allowdisplaybreaks
\begin{align}
r(x,y)\nu(dx)\nu(dy) & = r(y,x)\nu(dx)\nu(dy),\label{result:abs_cont:identity1}\\
r(x,y)r(y,z)\nu(dy)\nu(dz)) & = r(x,z)r(z,y)\nu(dy)\nu(dz)\qquad\mbox{for }\nu\mbox{-a.e. }x,\label{result:abs_cont:identity2}\\
r(x,y)r(y,z)\nu(dx)\nu(dy)\nu(dz)) & = r(x,y)r(x,z)\nu(dx)\nu(dy)\nu(dz).\label{result:abs_cont:identity3}
\end{align}
\endgroup
In particular, equations \eqref{result:abs_cont:identity1} and \eqref{result:abs_cont:identity2} imply that
\begingroup\allowdisplaybreaks
\begin{align}
\mathbbm{1}_{D_x}(y)\mathbbm{1}_{D_x}(z)r(x,y)\nu(dx)\nu(dy)\nu(dz)&\overset{(a)}{=}\mathbbm{1}_{D_x}(y)\mathbbm{1}_{D_x}(z)\mathbbm{1}_{D_y}(z)r(x,y)\nu(dx)\nu(dy)\nu(dz) \nonumber\\
&=\mathbbm{1}_{D_x}(y)\mathbbm{1}_{D_x}(z)\mathbbm{1}_{D_y}(z)r(x,y)\frac{r(y,z)}{r(y,z)}\nu(dx)\nu(dy)\nu(dz) \nonumber\\
&=\mathbbm{1}_{D_x}(y)\mathbbm{1}_{D_x}(z)\mathbbm{1}_{D_y}(z)r(x,y)\frac{r(z,y)}{r(y,z)}\nu(dx)\nu(dy)\nu(dz) \nonumber\\
&=\mathbbm{1}_{D_x}(y)\mathbbm{1}_{D_x}(z)\mathbbm{1}_{D_y}(z)r(x,z)\nu(dx)\nu(dy)\nu(dz) \nonumber\\
&\overset{(b)}{=}\mathbbm{1}_{D_x}(y)\mathbbm{1}_{D_x}(z)r(x,z)\nu(dx)\nu(dy)\nu(dz), \label{proof:abs_cont:equation1}
\end{align}
\endgroup
where $(a)$ and $(b)$ follow from the fact that, for $(\nu\times\nu\times\nu)$-a.e. $(x,y,z)$ in $\mathbb{X}^3$, if $y\in D_x$ and $z\in D_x$, then $r(x,y)r(x,z)>0$, and so $r(x,y)r(y,z)>0$ by \eqref{result:abs_cont:identity3}, which implies that $r(y,z)>0$ and, as a result, $\mathbbm{1}_{D_x}(y)\mathbbm{1}_{D_x}(z)=\mathbbm{1}_{D_x}(y)\mathbbm{1}_{D_x}(z)\mathbbm{1}_{D_y}(z)$.

Let $A,B\in\mathcal{X}$. Since $\nu(G)=1$,
\begingroup\allowdisplaybreaks
\begin{align}
\int_AR_x(B)\nu(dx)&=\int_A\int_B\mathbbm{1}_{D_x}(y)\cdot r(x,y)\nu(dy)\nu(dx) \nonumber\\
&=\int_A\mathbbm{1}_G(x)\biggl(\int_B\mathbbm{1}_{D_x}(y)\cdot r(x,y)\frac{1}{\nu(D_x)}\Bigl(\int_\mathbb{X}\mathbbm{1}_{D_x}(z)\nu(dz)\Bigr)\nu(dy)\biggr)\nu(dx) \nonumber\\
&\overset{(a)}{=}\int_A\mathbbm{1}_G(x)\biggl(\int_B\frac{1}{\nu(D_x)}\mathbbm{1}_{D_x}(y)\Bigl(\int_\mathbb{X}\mathbbm{1}_{D_x}(z)\cdot r(x,z)\nu(dz)\Bigr)\nu(dy)\biggr)\nu(dx) \nonumber\\
&=\int_A\mathbbm{1}_G(x)\biggl(\int_B\frac{1}{\nu(D_x)}\mathbbm{1}_{D_x}(y)\nu(dy)\biggr)\nu(dx) \nonumber\\
&=\int_A\nu(B|D_x)\nu(dx),\label{proof:abs_cont:conditional_final}
\end{align}
\endgroup
where $(a)$ follows from \eqref{proof:abs_cont:equation1}.

Given \eqref{proof:abs_cont:conditional_final}, the proof of Theorem \ref{result:abs_cont:main} would be complete if we can find a countable partition among the sets $D_x$. As we demonstrate next, this is possible, provided we intersect each $D_x$ with some a.s. sets. 

First, observe
\begingroup\allowdisplaybreaks
\begin{align}
1=\int_\mathbb{X}\int_\mathbb{X}R_y(\mathbb{X})R_x(dy)\nu(dx)&=\int_\mathbb{X}\int_{D_x}\Bigl(\int_{D_y}r(y,z)\nu(dz)\Bigr)r(x,y)\nu(dy)\nu(dx) \nonumber\\
&=\int_\mathbb{X}\int_{D_x}\Bigl(\int_{D_x\cap D_y}r(y,z)\nu(dz)\Bigr)r(x,y)\nu(dy)\nu(dx),\label{proof:abs_cont:partition:identity1}
\end{align}
\endgroup
where the last equality follows, similarly to \eqref{proof:abs_cont:equation1}, from the fact that $\mathbbm{1}_{D_x}(y)\mathbbm{1}_{D_y}(z)=\mathbbm{1}_{D_x}(y)\mathbbm{1}_{D_x}(z)\mathbbm{1}_{D_y}(z)$ for $(\nu\times\nu\times\nu)$-a.e. $(x,y,z)$ in $\mathbb{X}^3$. Since $\int_{D_x}r(x,y)\nu(dy)=1$ for $\nu$-a.e. $x$, the integrand in \eqref{proof:abs_cont:partition:identity1} w.r.t. $\nu(dx)$ is smaller or equal to $1$, so that
\[\int_{D_x}r(x,y)\Bigl(\int_{D_x\cap D_y}r(y,z)\nu(dz)\Bigr)\nu(dx)=1\qquad\mbox{for }\nu\mbox{-a.e. }x,\]
which itself implies that, for $\nu$-a.e. $y$ in $D_x$,
\[\int_{D_x\cap D_y}r(y,z)\nu(dz)=1.\]
But $\int_{D_y}r(y,z)\nu(dz)=1$, thus $\nu(D_x^c\cap D_y)=0$ for $\nu$-a.e. $y$ in $D_x$. On the other hand, 
\begingroup\allowdisplaybreaks
\begin{align}
\begin{aligned}\label{proof:abs_cont:partition:identity2}
0&=\int_\mathbb{X}\int_{D_x}\Bigl(\int_{D_y^c}r(y,z)\nu(dz)\Bigr)r(x,y)\nu(dy)\nu(dx)\\
&\overset{(a)}{=}\int_\mathbb{X}\int_{D_x}\Bigl(\int_{D_y^c}r(x,z)\nu(dz)\Bigr)r(x,y)\nu(dy)\nu(dx)\overset{(b)}{=}\int_\mathbb{X}\int_{D_x}\Bigl(\int_{D_x\cap D_y^c}r(x,z)\nu(dz)\Bigr)r(x,y)\nu(dy)\nu(dx),
\end{aligned}
\end{align}
\endgroup
where $(a)$ comes from \eqref{result:abs_cont:identity3}, and $(b)$ since $r(x,z)=0$ for $z\in D_x^c$. Applying the same reasoning as in \eqref{proof:abs_cont:partition:identity1} to \eqref{proof:abs_cont:partition:identity2}, we can show that, for $\nu$-a.e. $x$, it holds $\nu(D_x\cap D_y^c)=0$ for $\nu$-a.e. $y$ in $D_x$. 

It follows from both results that, for some $E\in\mathcal{X}:\nu(E)=1$ and every $x\in E$, there exists a further set $E_x\in\mathcal{X}$ such that $\nu(E_x)=1$ and, for all $y\in E_x\cap D_x$, it holds
\[\nu(D_x^c\cap D_y)=0\qquad\mbox{and}\qquad\nu(D_x\cap D_y^c)=0,\]
implying that
\begin{equation}\label{proof:abs_cont:partition:equivalence}
\nu(D_x)=\nu(D_x\cap D_y)=\nu(D_y).
\end{equation}

Recall that $G=\{x\in\mathbb{X}:\nu(D_x)>0\}$ and $\nu(G)=1$. Let us define, for $x\in E$,
\[\bar{D}_x:=E_x\cap D_x\qquad\mbox{and}\qquad\mathbb{X}_0:=G\cap E\cap\{x\in\mathbb{X}:R_x(dy)=r(x,y)\nu(dy)\}.\]
Then $\nu(\bar{D}_x)=\nu(D_x)>0$ for all $x\in\mathbb{X}_0$, and $\nu(\mathbb{X}_0)=1$. We will now show that the sets $\bar{D}_x$ and $\bar{D}_y$, for different $x,y\in\mathbb{X}_0$, are either disjoint or their difference has $\nu$-measure zero, which will allow us to create the aforementioned partition on $\mathbb{X}$.

Let $x,y\in\mathbb{X}_0$ be such that $\bar{D}_x\cap\bar{D}_y\neq\emptyset$. Pick $z\in\bar{D}_x\cap\bar{D}_y$. Then $z\in E_x$ and $z\in E_y$, so $\nu(D_x\cap D_z^c)=0$ and $\nu(D_y^c\cap D_z)=0$ from \eqref{proof:abs_cont:partition:equivalence}, which implies that
\[\nu(\bar{D}_x)=\nu(D_x)=\nu(D_x\cap D_y)+\nu(D_x\cap D_y^c\cap D_z)+\nu(D_x\cap D_y^c\cap D_z^c)=\nu(D_x\cap D_y)=\nu(\bar{D}_x\cap\bar{D}_y),\]
and, similarly, 
\begin{equation}\label{proof:abs_cont:partition_difference}
\nu(\bar{D}_y)=\nu(\bar{D}_x\cap\bar{D}_y)=\nu(\bar{D}_x).
\end{equation}
If there exists, in addition, $y'\in\mathbb{X}_0$ such that $\bar{D}_y\cap\bar{D}_{y'}\neq\emptyset$, then $\nu(\bar{D}_{y'})=\nu(\bar{D}_y\cap\bar{D}_{y'})$ from \eqref{proof:abs_cont:partition_difference} and
\[\nu(\bar{D}_x\cap\bar{D}_{y'})=\nu(\bar{D}_x\cap\bar{D}_y\cap\bar{D}_{y'})=\nu(\bar{D}_x\cap\bar{D}_y\cap\bar{D}_{y'})+\nu(\bar{D}_x^c\cap\bar{D}_y\cap\bar{D}_{y'})=\nu(\bar{D}_y\cap\bar{D}_{y'})=\nu(\bar{D}_{y'})>0;\]
thus, $\bar{D}_x\cap\bar{D}_{y'}\neq\emptyset$. Therefore, we can partition $\mathbb{X}_0$ into equivalent classes according to whether $\bar{D}_x\cap\bar{D}_y\neq\emptyset$ or not, and we can pick one element from each class to create a family of subsets $\{\bar{D}_x\}_{x\in\mathbb{Y}}$, for some $\mathbb{Y}\subseteq\mathbb{X}_0$, such that $\bar{D}_x\cap \bar{D}_y=\emptyset$, for $x\neq y$, and $\nu(\bar{D}_x)>0$, for all $x\in\mathbb{Y}$. In fact, $\mathbb{Y}$ is at most countable as $\nu$ is finite.

If $x\in\mathbb{X}_0$, then $R_x\ll\nu$, and thus $R_x(\bar{D}_x)=R_x(D_x)=1$; hence we get from \eqref{proof:abs_cont:partition_difference} that, for every $x,y\in\mathbb{X}_0$,
\[
R_x(\bar{D}_y)=R_x(\bar{D}_x\cap\bar{D}_y)=R_x(\bar{D}_x)=1,\qquad\mbox{if }\bar{D}_x\cap\bar{D}_y\neq\emptyset,\]
and
\[R_x(\bar{D}_y)=R_x(\bar{D}_x^c\cap\bar{D}_y)=0,\qquad\mbox{if }\bar{D}_x\cap\bar{D}_y=\emptyset.\]
The two equations together imply that $\sum_{x'\in\mathbb{Y}}R_x(\bar{D}_{x'})=1$ for each $x\in\mathbb{X}_0$, so using \eqref{proof:abs_cont:partition:identity2} and the fact that $\{\bar{D}_{x'}\}_{x'\in\mathbb{Y}}$ are disjoint, 
\[\nu\Bigl(\bigcup_{x'\in\mathbb{Y}}\bar{D}_{x'}\Bigr)=\sum_{x'\in\mathbb{Y}}\nu(\bar{D}_{x'})=\sum_{x'\in\mathbb{Y}}\int_\mathbb{X}R_x(\bar{D}_{x'})\nu(dx)=\int_{\mathbb{X}_0}\sum_{x'\in\mathbb{Y}}R_x(\bar{D}_{x'})\nu(dx)=\nu(\mathbb{X}_0)=1.\]

Let $A,B\in\mathcal{X}$. Then
\begingroup\allowdisplaybreaks
\begin{align*}
\int_AR_x(B)\nu(dx)&=\int_{A\cap(\bigcup_{x'\in\mathbb{Y}}\bar{D}_{x'})}R_x(B)\nu(dx)\\
&=\sum_{x'\in\mathbb{Y}}\int_{A\cap\mathbb{X}_0\cap\bar{D}_{x'}}R_x(B)\nu(dx)\\
&\overset{(a)}{=}\sum_{x'\in\mathbb{Y}}\int_{A\cap\mathbb{X}_0\cap\bar{D}_{x'}}\nu(B|D_x)\nu(dx)\\
&\overset{(b)}{=}\sum_{x'\in\mathbb{Y}}\int_{A\cap\mathbb{X}_0\cap\bar{D}_{x'}}\nu(B|\bar{D}_{x'})\nu(dx)=\int_A\sum_{x'\in\mathbb{Y}}\nu(B|\bar{D}_{x'})\cdot\mathbbm{1}_{\bar{D}_{x'}}(x)\nu(dx),
\end{align*}
\endgroup
where $(a)$ follows from \eqref{proof:abs_cont:conditional_final} with $A$ replaced by $A\cap\mathbb{X}_0\cap\bar{D}_{x'}$, and $(b)$ from \eqref{proof:abs_cont:partition:equivalence} since $\nu(\bar{D}_x^c\cap\bar{D}_{x'})=0=\nu(\bar{D}_x\cap\bar{D}_{x'}^c)$ whenever $x\in\bar{D}_{x'}\cap\mathbb{X}_0$. Therefore, $R_x(B)=\sum_{x'\in\mathbb{Y}}\nu(B|\bar{D}_{x'})\cdot\mathbbm{1}_{\bar{D}_{x'}}(x)$ for $\nu$-a.e. $x$ and, as $\mathcal{X}$ is countably generated, as measures,
\[R_x(\cdot)=\sum_{x'\in\mathbb{Y}}\nu(\cdot\mid\bar{D}_{x'})\cdot\mathbbm{1}_{\bar{D}_{x'}}(x)\qquad\mbox{for }\nu\mbox{-a.e. }x.\]
\end{proof}


\begin{proof}[Proof of Theorem \ref{result:decomposition:main}]
If $(X_n)_{n\geq1}$ is i.i.d., then Proposition \ref{result:i.i.d} implies that $R_x=R_x^a$, for $\nu$-a.e. $x$. Suppose that $(X_n)_{n\geq0}$ is exchangeable, but not i.i.d. Assume, without loss of generality, that $R_x(\mathbb{X})=1$ for $x\in\mathbb{X}$. By Lebesgue decomposition (Theorem 1 in \cite{lange1973}),
\[R_x=R_x^\perp+R_x^a,\]
for some finite transition kernels $R^\perp$ and $R^a$ on $\mathbb{X}$ such that $R_x^\perp\perp\nu$ for some $S_x\in\mathcal{X}$ with $\nu(S_x)=0=R_x^\perp(S_x^c)$, and $R_x^a\ll\nu$, for $x\in\mathbb{X}$. Moreover,
\begin{equation}\label{proof:general:decomposition:identity1}
R_x^\perp(\mathbb{X})=R_x^\perp(S_x)=R_x(S_x)\qquad\mbox{and}\qquad R_x^a(\mathbb{X})=R_x^a(S_x^c)=R_x(S_x^c).
\end{equation}
Arguing as in \eqref{proof:abs_cont:density}, there exists some measurable function $r:\mathbb{X}^2\rightarrow\mathbb{R}_+$ such that  
\begin{equation}\label{proof:general:decomposition:density}
R_x^a(dy)=r(x,y)\nu(dy)\qquad\mbox{for }\nu\mbox{-a.e. }x.
\end{equation}

In addition, we have from \eqref{proof:general:representation:identity1} that
\begin{equation}\label{proof:general:decomposition:condition1}
R_x^\perp(dy)\nu(dx)+R_x^a(dy)\nu(dx)=R_y^\perp(dx)\nu(dy)+R_y^a(dx)\nu(dy),
\end{equation}
and, from \eqref{proof:general:representation:identity2}, for $\nu$-a.e. $x$,
\begingroup\allowdisplaybreaks
\begin{equation}\label{proof:general:decomposition:condition2}
\begin{aligned}
R_y^\perp(dz)R_x^\perp(dy)&+R_y^a(dz)R_x^\perp(dy)+R_y^\perp(dz)R_x^a(dy)+R_y^a(dz)R_x^a(dy)\\
&=R_z^\perp(dy)R_x^\perp(dz)+R_z^a(dy)R_x^\perp(dz)+R_z^\perp(dy)R_x^a(dz)+R_z^a(dy)R_x^a(dz).
\end{aligned}
\end{equation}
\endgroup
Fix one such $x$. Since $\nu(S_x)=0$, then $R_y^a(S_x)=0$ for all $y\in\mathbb{X}$, and integrating \eqref{proof:general:decomposition:condition2} on $z\in S_x$ and $y\in\mathbb{X}$, we get
\begin{equation}\label{proof:general:decomposition:result}
\begin{aligned}
\int_\mathbb{X}R_y^\perp(S_x)R_x^\perp(dy)+\int_\mathbb{X}R_y^\perp(S_x)R_x^a(dy)&=\int_{S_x}R_z^\perp(\mathbb{X})R_x^\perp(dz)+\int_{S_x}R_z^a(\mathbb{X})R_x^\perp(dz)=R_x^\perp(\mathbb{X}),
\end{aligned}
\end{equation}
where we have used in the last equality that $R_x^\perp(S_x^c)=0$. On the other hand, from \eqref{proof:general:decomposition:density},
\begingroup\allowdisplaybreaks
\begin{align*}
\int_\mathbb{X}R_y^\perp(S_x)R_x^a(dy)&=\int_{S_x^c}R_y^\perp(S_x)R_x^a(dy)\\
&=\int_{y\in S_x^c}\int_{z\in S_x}r(x,y)R_y^\perp(dz)\nu(dy)\\
&\begin{aligned}\overset{(a)}{=}\int_{z\in S_x}\int_{y\in S_x^c}r(x,y)R_z^\perp(dy)\nu(dz)&+\int_{z\in S_x}\int_{y\in S_x^c}r(x,y)R_z^a(dy)\nu(dz)\\
&-\int_{y\in S_x^c}\int_{z\in S_x}r(x,y)R_y^a(dz)\nu(dy)\end{aligned}\\
&\overset{(b)}{=}0,
\end{align*}
\endgroup
where $(a)$ follows from \eqref{proof:general:decomposition:condition1}, and for $(b)$ we have used again that $\nu(S_x)=0$ and $R_y^a(S_x)=0$ for all $y\in\mathbb{X}$.

Plugging this into \eqref{proof:general:decomposition:result}, we get $\int_\mathbb{X}R_y^\perp(S_x)R_x^\perp(dy)=R_x^\perp(\mathbb{X})$ for $\nu$-a.e. $x$, and thus
\begin{equation}\label{proof:decomposition:singular:support}
R_y^\perp(S_x)=1\qquad\mbox{for }R_x^\perp\mbox{-a.e. }y.
\end{equation}
Moreover, $1=R_y^\perp(S_x)\leq R_y^\perp(\mathbb{X})=R_y(S_y)$, so $R_y^a(\mathbb{X})=0$, by \eqref{proof:general:decomposition:identity1}, and $R_y^a(dz)R_x^\perp(dy)=0$ for $\nu$-a.e. $x$, which implies the splitting of the measures
\[R_y(dz)R_x(dy)\nu(dx)=R_y^\perp(dz)R_x^\perp(dy)\nu(dx)+R_y(dz)R_x^a(dy)\nu(dx).\]
On the other hand, using the identity in \eqref{proof:general:equality}, we obtain
\[R_y(dz)R_x(dy)\nu(dx)=R_x(dz)R_x(dy)\nu(dx)=R_x(dz)R_x^\perp(dy)\nu(dx)+R_x(dz)R_x^a(dy)\nu(dx).\]
Combining the two result, we get, for $\nu$-a.e. $x$,
\[R_y^\perp(dz)R_x^\perp(dy)+R_y(dz)R_x^a(dy)=R_x(dz)R_x^\perp(dy)+R_x(dz)R_x^a(dy).\]
But $R_x^\perp(dy)\perp R_x^a(dy)$ w.r.t. $S_x$ and $R_x^\perp(S_x)=1$, so it must be
\[R_y^\perp(dz)R_x^\perp(dy)=\mathbbm{1}_{S_x}(y)R_y^\perp(dz)R_x^\perp(dy)=\mathbbm{1}_{S_x}(y)R_x(dz)R_x^\perp(dy)=R_x(dz)R_x^\perp(dy).\]
It follows from this result, \eqref{proof:general:decomposition:identity1} and \eqref{proof:decomposition:singular:support} that, for $\nu$-a.e. $x$,
\[R_x(S_x)=R_x^\perp(S_x)=\int_{S_x}R_y^\perp(S_x)R_x^\perp(dy)=\int_{S_x}\int_{S_x}R_y^\perp(dz)R_x^\perp(dy)=\int_{S_x}R_x(S_x)R_x^\perp(dy)=R_x(S_x)R_x(S_x);\]
thus, $R_x(S_x)\in\{0,1\}$ for $\nu$-a.e. $x$.

Let us define
\[\mathcal{S}:=\{x\in\mathbb{X}:R_x^\perp(\mathbb{X})=1\},\]
which is $\mathcal{X}$-measurable since $R_x^\perp$ is a transition kernel. By \eqref{proof:general:decomposition:identity1}, we have $R_x(S_x)=R_x^\perp(\mathbb{X})=1$, for $x\in\mathcal{S}$. Therefore, using that $R_x(S_x)\in\{0,1\}$ for $\nu$-a.e. $x$, we get for $\nu$-a.e. $x$ that: if $x\in\mathcal{S}$, then $R_x=R_x^\perp$; if $x\in\mathcal{S}^c$, then $R_x=R_x^a$. Moreover, \eqref{proof:general:decomposition:identity1} and \eqref{proof:decomposition:singular:support} imply for $\nu$-a.e. $x$ that $y\in\mathcal{S}$ for $R_x^\perp$-a.e. $y$; in other words, $R_x^\perp(\mathcal{S}^c)=0$ for $\nu$-a.e. $x$. Finally, using \eqref{proof:general:representation:identity1} and the fact that $R_x^a=0$ for $\nu$-a.e. $x\in\mathcal{S}$ and $R_x^\perp=0$ for $\nu$-a.e. $x\in\mathcal{S}^c$, we get
\begingroup\allowdisplaybreaks
\begin{align*}
\int_\mathbb{X}R_x^a(\mathcal{S})\nu(dx)&=\int_\mathcal{S}R_x^a(\mathcal{S})\nu(dx)+\int_{\mathcal{S}^c}R_x^a(\mathcal{S})\nu(dx)=\int_{\mathcal{S}^c}R_x(\mathcal{S})\nu(dx)\\
&=\int_\mathcal{S}R_x(\mathcal{S}^c)\nu(dx)=\int_\mathcal{S}R_x^\perp(\mathcal{S}^c)\nu(dx)+\int_{\mathcal{S}^c}R_x^\perp(\mathcal{S}^c)\nu(dx)=\int_\mathbb{X}R_x^\perp(\mathcal{S}^c)\nu(dx)=0,
\end{align*}
\endgroup
and so $R_x^a(\mathcal{S})=0$ for $\nu$-a.e. $x$.
\end{proof}

\begin{proof}[Proof of Proposition \ref{result:singular:discrete:atoms}]
Since $R_x\perp\nu$, then $(X_n)_{n\geq1}$ is not i.i.d., and hence not unbalanced by Theorem \ref{result:general:balance}, so from Corollary \ref{result:general:balance_probability} we may assume that $R_x(\mathbb{X})=1$, without loss of generality.

Let us define, for $x\in\mathbb{X}$,
\[D_x:=\bigl\{y\in\mathbb{X}:R_x(\{y\})>0\bigr\}\qquad\mbox{and}\qquad D:=\{x\in\mathbb{X}:x\in D_x\}.\]
Then $R_x(\cdot)=\sum_{y\in D_x}p_x(y)\delta_y(\cdot)$ for some $p_x(y)>0$ such that $\sum_{y\in D_x}p_x(y)=1$.

It follows from \eqref{proof:general:equality} and the form of $R_x$ that there exists $C\in\mathcal{X}$ such that $\nu(C)=1$ and, for all $x\in C$,
\[R_x(dz)=R_y(dz)\qquad\mbox{for }y\in D_x;\]
thus, $D_x=D_y$ for all $y\in D_x$, which implies that $y\in D_y$, i.e. $y\in D$. Therefore, by \eqref{proof:general:representation:identity1}, 
\[\nu(D)=\int_\mathbb{X}R_x(D)\nu(dx)=\int_C\Bigl(\sum_{y\in D_x}p_x(y)\delta_y(D)\Bigr)\nu(dx)=\int_C\Bigl(\sum_{y\in D_x}p_x(y)\Bigr)\nu(dx)=1.\]

For the second part, let $x,y\in C\cap D$ be such that $y\notin D_x$. Assume $D_x\cap D_y\neq\emptyset$. Take $z\in D_x\cap D_y$. It follows from above that $D_x=D_z=D_y$. But $y\in D$, so $y\in D_y=D_x$, absurd. As a result, for $\nu$-a.e. $x$ and $y$, either $D_x=D_y$ or $D_x\cap D_y=\emptyset$.
\end{proof}

\section{Discussion}\label{section:discussion}

The results in this paper allow us to state some universal facts about exchangeable MVPSs that were previously unknown: 1) all models are necessarily balanced; 2) the normalized reinforcement kernels are a.e. proper regular conditional distributions; 3) the prior distributions have the stick-breaking construction of a Dirichlet process. When $\mathbb{X}$ is countable or the reinforcement is dominated by $\nu$, then every MVPS is a Dirichlet process mixture model over a family of probability distributions with disjoint supports derived from $\nu$. Relaxing parts of this structure while retaining exchangeability would lead to a very different sampling scheme from the one that \eqref{intro:predictive:mvps} entails. On the other hand, the fact that, for fixed $x$, $R_x$ is either absolutely continuous or mutually singular with respect to $\nu$, though surprising at first, seems very natural under exchangeability. Therefore, we expect that $R$ can be further decomposed along other lines, e.g., $R_x^\perp$ might be discrete only for some $x$ and diffuse for the rest.

\subsection*{Acknowledgments}

This study is financed by the European Union-NextGenerationEU, through the National Recovery and Resilience Plan of the Republic of Bulgaria, project No. BG-RRP-2.004-0008.

We are grateful to an Associate Editor and two anonymous referees for their valuable comments and helpful suggestions, and for pointing out inaccuracies in an earlier version of this work.

\bibliography{mybib}

\end{document}